\documentclass[12pt]{article}
\ifdefined\kanjiskip
  \PassOptionsToPackage{dvipdfmx}{color}
  \PassOptionsToPackage{dvipdfmx}{graphicx}
  \PassOptionsToPackage{dvipdfmx}{hyperref}
\fi
\usepackage{mathrsfs,amssymb,amsmath,
amsthm, color, graphicx }
\usepackage{bm}
\usepackage{url} 
\usepackage{hyperref} 
\IfFileExists{orcidlink.sty}{%
  \usepackage{orcidlink}%
}{%
  \definecolor{orcidlogocol}{RGB}{166,206,57}%
  \DeclareRobustCommand{\orcidlink}[1]{%
    \texorpdfstring{%
      \href{https://orcid.org/##1}{%
        \raisebox{-0.10ex}{%
          \textcolor{orcidlogocol}{%
            \textcircled{\raisebox{0.10ex}{\scriptsize\sffamily\bfseries iD}}%
          }%
        }%
      }%
    }{ORCID~##1}%
  }%
}
\usepackage{appendix}
\newtheorem{tm}{Theorem}[section]
\newtheorem{lm}[tm]{Lemma}
\newtheorem{co}[tm]{Corollary}
\newtheorem{re}[tm]{Remark}
\newtheorem{df}[tm]{Definition}
\newtheorem{exm}[tm]{Example}
\newtheorem{pr}[tm]{Proposition}

\newcommand{\mc}[1]{{\mathcal #1}}
\newcommand{\mf}[1]{{\mathfrak #1}}

\newcommand{\bb}[1]{{\mathbb #1}}

\newcommand{\E}{\mathbb{E}}

\newcommand{\1}{\mathbf{1}}
\newcommand{\Mplus}{\mathcal{M}_+(M)}

\newcommand{\R}{\mathbb{R}}
\newcommand{\supp}{\operatorname{supp}}

\newcommand{\lan}{\langle}
\newcommand{\ran}{\rangle}
\renewcommand{\hat}{\widehat}
\renewcommand{\tilde}{\widetilde}

\allowdisplaybreaks[4]

\makeatletter
\newcommand{\subscripts}[3]{%
  \@mathmeasure\z@\displaystyle{#2}%
  \global\setbox\@ne\vbox to\ht\z@{}\dp\@ne\dp\z@
  \setbox\tw@\box\@ne
  \@mathmeasure4\displaystyle{\copy\tw@_{#1}}%
  \@mathmeasure6\displaystyle{{#2}_{#3}}%
  \dimen@-\wd6 \advance\dimen@\wd4 \advance\dimen@\wd\z@
  \hbox to\dimen@{}\mathop{\kern-\dimen@\box4\box6}%
}
\makeatother

\newcommand{\III}{{\vert \kern-.10em \vert \kern-.10em \vert}}

\makeatletter
 
 \@addtoreset{equation}{section}
\makeatother
 
\begin{document}
\setlength{\baselineskip}
{15.5pt}
%
\allowdisplaybreaks

\title{
Hydrodynamic limit of 
exclusion processes with 
\\
killing
on weighted 
Riemannian 
manifolds\\
via a graph discretization
}
\author{\large
{Satoshi Ishiwata\hspace{1mm}\footnote{Department of Mathematical Sciences, 
Faculty of Science, Yamagata University,
1-4-12, Kojirakawa, Yamagata 990-8560, Japan
(e-mail: {\tt ishiwata@sci.kj.yamagata-u.ac.jp})},\,
Hiroshi Kawabi\,\orcidlink{0009-0006-7128-206X}
\hspace{-1mm}\footnote{Department of Mathematics, Hiyoshi Campus, Keio University,
4-1-1, Hiyoshi, Kohoku-ku, Yokohama, Kanagawa 223-8521, Japan
(e-mail: {\tt{kawabi@keio.jp}})}\,
and Kenkichi Tsunoda\hspace{1mm}\footnote{Faculty of Mathematics, Kyushu University,
744 Motooka, Nishi-ku, Fukuoka 819-0395, Japan (e-mail: {\tt{tsunoda@math.kyushu-u.ac.jp}})}
}}
\maketitle
\begin{abstract}
In the present paper, we consider an exclusion process with killing
on a proximity graph constructed from a partition of a geodesically 
complete weighted Riemannian manifold.
We rigorously derive its hydrodynamic equation, which is governed by 
a heat equation with a killing potential term on the manifold.
Within our graph discretization framework, we define an empirical density field
for the exclusion process and prove its convergence, under an appropriate space-time scaling,
to the unique bounded 
weak solution of the hydrodynamic equation, provided that 
the weighted manifold is stochastically complete.
Combining local regularity for the parabolic equation with
techniques from stochastic analysis on manifolds, we prove 
the uniqueness of the bounded weak solution and give the explicit 
representation of the solution in terms of the minimal Schr{\"{o}}dinger kernel.
Our result requires neither a
global lower Ricci curvature bound nor global boundedness of the weight function. 
\vspace{0mm} 
\vspace{1mm} 
\end{abstract}
\section{Introduction}
The {\it{hydrodynamic limit}} for interacting particle systems plays a fundamental role 
in understanding
the relation between microscopic and macroscopic phenomena in 
statistical physics and hydrodynamics.
Roughly speaking, deriving the hydrodynamic limit means deducing the macroscopic
behavior of the system from its microscopic interacting-particle dynamics. It can be
regarded as a law of large numbers for the stochastic dynamics of many interacting
particles under an appropriate space-time scaling.
Since the limiting macroscopic behavior is 
governed by a deterministic evolution equation, 
which is 
called the {\it{hydrodynamic equation}}, 
the hydrodynamic limit is of interest
not only in physics but also in many branches of mathematics, such as
probability theory, partial differential equations, and geometric analysis.
Indeed, this theme has been studied extensively by many authors
(see e.g., \cite{Spo91, KL99, Sep08, Fun18} and references therein for further related results).

The \textit{exclusion process} is one of the most fundamental classes of
interacting particle systems. It describes a system of random walks on a
discrete state space subject to the \textit{exclusion rule}, namely, each site
can be occupied by at most one particle. In the standard setting of the lattice
$\mathbb Z^d$ with nearest-neighbor symmetric jumps, the hydrodynamic limit is
given, under diffusive space-time scaling, by the heat equation on Euclidean
space. Much of the classical theory has therefore been developed for particle
systems on $\mathbb Z^d$ or on the discrete torus $(\mathbb Z/N\mathbb Z)^d$,
with limiting partial differential equations posed on $\mathbb R^d$ or on the
flat torus $\mathbb T^d$.
We refer to \cite{KL99} and \cite{Sep08}
for the hydrodynamic limit of exclusion processes
on $\mathbb T^d$ and $\mathbb R^d$, respectively.

It is natural, both from physical motivation and from the viewpoint of geometry,
to investigate the scaling limits of interacting particle systems evolving on more
general spaces, such as Riemannian manifolds or fractals, and to clarify how the
macroscopic behavior of particles reflects the geometry of the underlying
space. Several works have addressed hydrodynamic limits on non-Euclidean or
inhomogeneous structures. 
Jara \cite{Jar09}, Chen and Gon\c{c}alves \cite{CG21} and
van Meurs and the third author \cite{vMT26} studied hydrodynamic limits
on the Sierpi\'{n}ski gasket, while
Tanaka \cite{Tan12} and Guan \cite{Gua23} studied exclusion processes on
crystal lattices using discrete geometric analysis initiated by Kotani and
Sunada \cite{KS00, Sun08, Sun13}. Hydrodynamic limits for symmetric exclusion processes in inhomogeneous
media were investigated by Faggionato \cite{Fag10} and Jara \cite{Jar11}. In
the Riemannian manifold setting, van Ginkel and Redig \cite{GR20, GR22} studied the
hydrodynamic limit and equilibrium fluctuations for the symmetric exclusion
process on closed manifolds. More recently, Junn\'{e}, Redig and Versendaal
\cite{JRV24} extended the hydrodynamic-limit result of \cite{GR20} to complete
non-compact manifolds. 

In the present paper, we consider an exclusion process with {\it{killing}} on a
connected, geodesically complete, possibly non-compact weighted Riemannian manifold $M$.
A fundamental difficulty is that a general manifold need not admit a compatible group action
and therefore has no canonical periodic structure.
Consequently, those $\mathbb Z^d$ results cannot be applied directly, and
even before taking the scaling limit, one has to choose a
suitable discretization of the manifold and formulate an exclusion process on
that discretized object. Following the earlier work \cite{IK24} of the
first two authors, we approximate $M$ by a sequence of connected proximity
graphs. Their vertices are constructed by partitions that cut $M$ into small
pieces, and their edges encode proximity between these pieces. 
Their vertices are the cells of partitions of $M$ and their edges encode proximity between these cells.
We then define a family of random walks with killing on these proximity
graphs and use them to construct the corresponding exclusion processes.
Microscopically, the jump to the cemetery state represents irreversible particle loss
caused by spatially inhomogeneous traps, absorbing defects, or sinks.
Under the diffusive scaling, the geometry and the weight determine transport through the weighted
Laplacian, whereas the local removal mechanics produces the potential term $-V(x)u$.
Therefore, this model describes diffusion with heterogeneous absorption on a curved and
inhomogeneous medium.

Our main result (Theorem \ref{mainthm}) states that, under the additional assumption that the weighted
manifold $M$ is {\it{stochastically complete}} (see Section \ref{manifold setting}), 
a family of empirical density fields of these
exclusion processes converges, under an appropriate space-time scaling, to the unique
bounded weak solution of the heat equation with a killing potential term on $M$.
This means that our result connects microscopic particle systems with annihilation to
macroscopic diffusion with killing on curved spaces.
Note that geodesic completeness alone does not in general imply stochastic completeness.
See e.g., Grigor'yan \cite{Gri09} for details.
The passage from the random walk approximation established in
\cite{IK24} to the interacting particle system considered here is not a
direct consequence of that work. Actually, several additional difficulties arise.
The transition probabilities of the random walk on the proximity graph are
generally non-symmetric. The graph is only locally finite, and its vertex
degrees need not be uniformly bounded. Moreover, the unweighted
particle-counting field does not have the normalization required for the
continuum limit. To overcome these difficulties, we construct an exclusion
process as a solution to a martingale problem. We also introduce a canonical
empirical density field whose site weights are determined by the reversible
measure associated with the random walk in the absence of
killing. This choice gives the correct continuum normalization. Through the
corresponding detailed-balance identity, it also transforms the quadratic
occupation term generated by the exclusion rule into an anti-symmetric sum,
which cancels exactly. The locally almost equi-partition condition controls
the remaining discrepancy between the weights of neighboring cells. Finally,
under the diffusive scaling, the microscopic killing mechanism appears in
the hydrodynamic equation as the potential term $-V(x)u$.

We now compare our result with the recent work  \cite{JRV24}, which also studies
hydrodynamic limits for exclusion processes on the non-compact weighted manifold $M$.
Although the macroscopic equation obtained there is closely related to ours, 
the microscopic construction is quite different. In \cite{JRV24}, the authors use random
neighborhood graphs whose vertices are given by a Poisson point process on $M$, and 
an edge connects two vertices whenever they are close. Equipping
these random graphs with symmetric edge weights, they consider a family of symmetric 
exclusion processes and finally derive the heat equation on $M$ as the hydrodynamic equation 
under a lower Ricci curvature bound, boundedness of the weights, and stochastic completeness. 
By contrast, our graph
approximation is based on the proximity graph discretization associated with
partitions of $M$, and even in the absence of killing,
our exclusion processes do not fall within the class of symmetric exclusion processes 
(see Remark \ref{rem-sym} for details).

There are also three further important differences. 
First, in our setting, there are no restrictions on the Ricci curvature or 
the weight of the manifold.
Second, the killing mechanism, which is
one of the main features of our model, is not treated in \cite{JRV24}.
It should be noted that 
the killing term can be removed by using Doob's $h$-transform (see \cite{Gri06})
in the macroscopic heat equation on the manifold. 
However, it is difficult to obtain the harmonic function $h$ explicitly in general, and hence 
the transform makes the bounded Cauchy problem complicated.
(Instead of the class of bounded functions, the Cauchy problem should be studied in the class
of functions bounded by a constant multiple of $h$.)
Moreover, as far as we know, 
there are no corresponding transforms in the microscopic discrete setting.
Third, in \cite{JRV24}, the authors impose in advance the uniqueness 
of the bounded weak solution of the heat equation to identify all limit
points of empirical density fields under an appropriate space-time scaling.
By contrast, in Section \ref{uni} below, we provide a self-contained proof
of the required uniqueness of bounded weak solutions under the
stochastic completeness assumption. This uniqueness result (Theorem \ref{thm:uniq})
allows us to complete the identification of the hydrodynamic
limit for the heat equation with the potential term $-V(x)u$.

The paper is organized as follows:
In Section 2, we introduce the setting of the paper. We first recall basic facts about weighted Riemannian manifolds and the associated heat and Schr\"{o}dinger kernels. We then review the proximity graph discretization introduced in \cite{IK24}. On the resulting graphs, we define random walks with killing and the transition operator for the exclusion dynamics. 
A key feature of this section is the introduction of three empirical density fields, namely the canonical empirical measure, an empirical measure in which each partition element is weighted by its volume under the square root of the given weight function, and the unweighted empirical measure. We explain the role of each field in the Riemannian setting and how its normalization is related to the corresponding continuum measure. We then formulate the hydrodynamic equation and introduce the notion of a bounded weak solution. At the end of the section, we state the uniqueness theorem 
for bounded weak solutions and present our main results on the hydrodynamic limit.

Section 3 is devoted to the proof of Theorem \ref{mainthm}. In Section 3.1, we establish preliminary estimates for the graph discretization and the discrete generator. In Section 3.2, we construct the exclusion process as a solution to the martingale problem for the transition operator $\mathcal L_{\alpha,\beta}$. This construction is needed because the proximity graph is only locally finite and the number of neighboring pieces is not assumed to be uniformly bounded. In Section 3.3, we derive Dynkin's martingale decomposition for the empirical density field and estimate its predictable quadratic variation. In Section 3.4, we prove tightness in the Skorokhod space $D([0,T],{\mathcal M}_{+}(M))$ and characterize every limit point as a bounded weak solution of the hydrodynamic equation. In Section 3.5, we complete the proof of Theorem \ref{mainthm} by applying the uniqueness result proved in Section 4. We also obtain the hydrodynamic limits for the other empirical density fields introduced in Section 2.

Section 4 provides a self-contained proof of the uniqueness of bounded weak solutions to the heat equation with killing. We first establish local parabolic regularity for bounded weak solutions. We then prove a Feynman--Kac type formula by using stochastic analysis on the weighted manifold. The stochastic completeness assumption allows us to remove the stopping times appearing in this formula. Finally, we prove Theorem \ref{thm:uniq} and obtain the representation of the solution in terms of the minimal Schr\"{o}dinger kernel.
\vspace{2mm} \\
\noindent
{\textbf{Notation.}} We collect here some basic notation which will be used frequently 
throughout the present paper.
\begin{itemize}
\vspace{-2mm}
\item[$\bullet$] We use $c$ and
$C$
to denote positive constants which may change from line to line. 
We also use the Landau symbols 
$O(\cdot)$ and $o(\cdot)$. If the dependence of $C$, $O(\cdot)$, and 
$o(\cdot)$ is significant, we indicate it by writing $C_A$, $O_{A}(\cdot)$, and $o_{A}(\cdot)$, 
respectively. Here $A$ may be a number or a subset of the space under consideration.
\vspace{-2mm}
\item[$\bullet$]
For a set $A$, we denote by $\sharp A$ the cardinality of $A$.
\vspace{-2mm}
\item[$\bullet$]
For $a\in \mathbb R$, we denote by $[a]$
the greatest integer less than or equal to $a$.

\vspace{-2mm}
\item[$\bullet$]
Unless otherwise specified, we use the Einstein summation convention, 
which means that an index variable that appears twice in an expression is 
implicitly summed over all its possible values.
\vspace{-2mm}
\item[$\bullet$]
For a fixed time $T>0$ and a metric space $S$, we denote by 
$C([0,T], S)$ the space of continuous $S$-valued maps defined on $[0,T]$.
Moreover, we denote by $D([0,T], S)$ the space of $S$-valued maps defined on 
$[0,T]$ which are right-continuous with left limits.
\vspace{-2mm}
\item[$\bullet$] For the geodesically complete Riemannian manifold $M$, we denote by
$B_{b}(M)$, $C_{0}(M)$, $C_{c}(M)$ and $C^{\infty}_{c}(M)$ the
spaces of bounded Borel measurable functions, continuous functions vanishing at infinity, 
continuous functions with compact support, and 
smooth functions with compact support on $M$, respectively. 
These spaces will be equipped with the sup norm $\Vert \cdot \Vert_{\infty}$.
\vspace{-2mm}
\item[$\bullet$] For a Borel measure $\mu$ on $M$ and $p\geq 1$, we denote by $L^{p}(\mu)(=L^{p}(M,\mu))$
the usual $L^{p}$-space associated with $\mu$ equipped with the $L^{p}$-norm $\Vert \cdot \Vert_{L^{p}(\mu)}$.
\vspace{-2mm}
\item[$\bullet$] We denote by ${\mathcal M}_{+}(M)$ 
the set of non-negative Radon measures on $M$
equipped with the vague topology. 
Here we say that a sequence $\{\pi_{n} \}_{n=1}^{\infty} \subset
{\cal M}_{+}(M)$ converges to $\pi \in {\cal M}_{+}(M)$ in the vague topology if
$\lim_{n\to \infty} \langle \pi_{n}, \varphi \rangle=
\langle \pi, \varphi \rangle$ for all
$\varphi\in C_{c}(M)$, where $\langle \pi, \varphi \rangle:= \int_M \varphi(x) \pi(dx)$.
\end{itemize}
\section{Model and main results}
In this section, we first fix the notation of differential geometric objects and review a graph discretization introduced 
in \cite{IK24}. After these preparations, we introduce exclusion processes with killing and state our main results.
\subsection{Weighted Riemannian manifold}
\label{manifold setting}
Let $M$ be a smooth $d$-dimensional 
Riemannian manifold equipped with a Riemannian metric $g$.
We assume that $M$ is geodesically complete and connected, but not necessarily compact.
Let $\{\mu_{\alpha} \}_{\alpha \geq 0}$ be a family of Borel measures on $M$ defined
by $\mu_{\alpha}(dx):=w^{\alpha}(x) {\rm vol}_{g}(dx)$,
where $w$ is a smooth positive function on $M$ and 
${\rm vol}_{g}$ is the Riemannian volume measure on $M$. 
We work with the framework of a {\it{weighted}} Riemannian manifold 
$M=(M, g, \mu_{\alpha})$ throughout the present paper.
For notational simplicity, we usually write it as $(M, \mu_{\alpha})$.
We denote by 
$d(x,y)$, $x,y\in M$, the geodesic distance between $x$ and $y$, 
and by $B(x,r)$ the open geodesic
ball of radius $r>0$ centered at $x\in {M}$. 
For any subset $A$ of $ M$, we define 
${\rm{diam}}(A):=\sup_{x,y\in A}d(x,y)$ and 
denote by $U_{r}(A)$ 
the $r$-neighborhood of $A $, i.e., 
\begin{equation*}
U_r(A)=\bigcup_{x \in A} B(x,r).
\end{equation*}
For two subsets $A$ and $B$ of 
${M}$, we define the {\it{Hausdorff distance}} $d_{H}(A,B)$
between $A$ and $B$ by
$$ d_{H}(A,B):=\inf \{ r>0;~A\subset U_{r}(B),~B \subset U_{r}(A) \}.$$

Let $\partial$ be a {\it{cemetery 
point}} added to 
$M$ so that 
$M_{\partial}:=M \cup \{ \partial \}$
is the one-point compactification of $M$. Since we regard $\partial$ as the point at infinity, 
we define $d_{H}(\{\partial \}, A):=\infty$ for any subset $A$ of $M$.

For any smooth function $f$ on $M$, we denote by $\nabla f$ its gradient vector field,
and for any smooth vector field $b$ on $M$, we define its {\it{weighted divergence}} 
${\rm div}_{\mu_{\alpha}} b$ by
$$ {\rm div}_{\mu_{\alpha}} b(x):=\frac{1}{w^{\alpha}(x)} {\rm div}(w^{\alpha}b)(x), \quad x\in M.$$
We then define 
the {\it{weighted Laplacian}} $\Delta_{\alpha}=\Delta_{\mu_{\alpha}}$ by
\begin{align}
\Delta_{\alpha} f(x)&:={\rm div}_{\mu_{\alpha}}(\nabla f)(x)
=
\Delta f(x)
+\frac{\alpha}{w(x)}g_{x}\big( \nabla w, \nabla f \big), \quad x\in M,
\end{align}
where 
$\Delta$ is the usual negative Laplacian (see e.g., \cite[Section 3.6]{Gri09} for details).

Let $p_\alpha (t,x, y)$, $t>0, x,y\in M$, be the {\it{minimal heat kernel}} on $M$, 
namely the minimal fundamental solution of the heat equation:
\begin{equation*}
\frac{\partial}{\partial t} u(t,x)=\Delta_\alpha u(t,x) , \quad u(0,x)=u_0(x), \quad t>0, \, x \in M.
\end{equation*}
It coincides with the transition function of the (weighted) {\it{Brownian motion}}
${\bf X}=(X_{t}, {\mathbb P}_{x}, \zeta)$ on the weighted manifold $(M, \mu_{\alpha})$
generated by $\Delta_{\alpha}$,
where ${\mathbb P}_{x}$ denotes the probability law of the underlying process
$(X_{t})$ starting from $x\in M$ and
$\zeta:=\inf \{t>0; \, X_{t} \notin M \}\in (0, \infty]$ is the {\it{lifetime}}.
We say that $(M, \mu_{\alpha})$ is \textit{stochastically complete} if  
\begin{equation}
\int_M p_\alpha (t,x, y) \mu_\alpha (dy)=1 \quad \mbox{for all }x\in M~\mbox{and }t>0.
\label{SC-kernel}
\end{equation}
The condition (\ref{SC-kernel}) is equivalent to 
\begin{equation}
{\mathbb P}_{x}(\zeta=\infty)=1
\quad \mbox{for all }x\in M.
\label{SC-life}
\end{equation}
See \cite[Lemma 8.5]{Gri06} for the proof.

Let $V$ be a killing potential, that is, a non-negative smooth function on $M$. 
We consider the Schr{\"o}dinger
operator $H_{V}^{(\alpha)}:=-\Delta_{\alpha}+V$ defined on $C^{\infty}_{c}(M)$. 
Since $M$ is geodesically complete, this differential operator
is essentially self-adjoint in $L^{2}(M, \mu_{\alpha })$ and 
the associated Schr{\"o}dinger semigroup $\{ e^{-tH_{V}^{(\alpha)}} \}_{t\geq 0}$
has an integral kernel $p^{V}_\alpha (t,x,y)$ with respect to $\mu_{\alpha}$, where
$p^{V}_\alpha (t,x,y)$ is a positive smooth function of $(t,x,y)\in (0,\infty)\times M \times M$
satisfying $p^{V}_{\alpha}(t,x,y)=p^{V}_{\alpha}(t,y,x)$.
We call $p^{V}_\alpha (t,x,y)$ the {\it{minimal Schr{\"o}dinger kernel function}}.
Clearly, it coincides with the minimal heat kernel $p_\alpha (t,x,y)$ in the case $V=0$.
%
\subsection{Random walk via graph discretization}
A countable collection $\mathbb{X} $ 
of connected measurable subsets of $M$ with finite positive measures
is called 
a \textit{partition} of ${M}$
if
\begin{equation}
M=\bigcup_{X \in {\mathbb X}} X~~\mbox{and }~
X \cap Y =\emptyset \hspace{3mm} \mbox{for all }
X,Y \in \mathbb{X} \mbox{ with }X \neq Y.
\label{bunkatsu-disjoint}
\end{equation}
Throughout the present paper, we assume
\begin{description}
\item[\bf{(A):}]~$\sharp  \hspace{0.5mm} \big \{ X \in \mathbb{X} ; X \subset B(x,r) \big \}<\infty$
for all $x\in M$ and $r>0$;
\item[\bf{(B):}]~$\vert {\mathbb X} \vert:= \sup_{X \in \mathbb{X}} {\rm diam} (X) <\infty.$
\end{description}
For a given $x\in M$,
we denote by $X (x)$ the unique element $X\in {\mathbb X}$ containing $x$.
For each $X \in \mathbb{X}$, we take a reference point $ x \in X$ and denote it by $x(X)$. 
We also set $X(\partial)=\{\partial \}$ and $x(\{\partial \})=\partial$.
We denote the set of reference points by $\mathscr X =\{ x(X) \}_{X \in \mathbb{X}}$.
It follows from (\ref{bunkatsu-disjoint}) that 
$x(X) \neq x(Y)$ for $X\neq Y$.
If the manifold $M$ is compact, condition {\bf{(A)}} implies that such a partition ${\mathbb X}$ is a finite set.
It is always possible to construct a partition 
$\mathbb X$ with {\bf{(A)}} and {\bf{(B)}}. See \cite[page 2464]{IK24} for details.

For a given $\rho>0$, let $\mathbb{X}={\mathbb X}({\rho})$ be a partition of $M$ with ${\bf{(A)}}$ and ${\bf{(B)}}$
satisfying $\vert \mathbb X \vert < \rho/3$. 
We say that $X$ and $Y$ in $\mathbb{X}$ are \textit{adjacent} if $d_H(X,Y)<\rho$ and write $X\sim_\rho Y$. 
Then we define an oriented graph $\mathbb{G}(\mathbb{X}, \rho)=(\mathbb{V}, \mathbb{E})$ called 
a ($\rho$-){\it{proximity graph}} of the manifold $M$ by $\mathbb{V}: ={\mathbb X}$ and 
$$\mathbb{E}:=\{ e=(X, Y)\in \mathbb{X}\times \mathbb{X}; X \sim_\rho Y \}.$$ 
For $X\in {\mathbb X}$, we define its ($\rho$-)neighborhood by 
$N_{\rho}(X):=\{ Y \in \mathbb{X} ; Y \sim_\rho X \}$.
Since $M$ is connected, ${\mathbb G}({\mathbb X}, \rho)$ is also connected, and
{\bf{(A)}} implies that ${\mathbb G}({\mathbb X}, \rho)$ is
locally finite, that is, 
$\sharp N_{\rho}(X)<\infty$. 
For later use, we introduce an enlarged
graph ${\mathbb G}^{\partial}({\mathbb X}, \rho)
=({\mathbb V}^{\partial}, {\mathbb E}^{\partial})$ 
obtained by adding the cemetery point $ \partial $ mentioned above to
the proximity graph $\mathbb G({\mathbb X}, \rho)$. 
To be precise,
${\mathbb V}^{\partial}={\mathbb X}^{\partial}:={\mathbb X} \cup \{X(\partial)\}$
and
$${\mathbb E}^{\partial}:={\mathbb E} \cup \{ (X, X(\partial)),  (X(\partial), X);  X \in \mathbb{X} \}
\cup \{ (X(\partial), X(\partial)) \}.$$ 
For an edge $e\in {\mathbb E}^{\partial}$, we denote by $o(e)$ and $t(e)$ the origin and the terminus of $e\in 
{\mathbb E}^{\partial}$, respectively. 
The inverse edge of $e \in {\mathbb E}^{\partial}$ is defined by the edge, say $\overline{e}$, satisfying
$o({\overline{e}})=t(e)$ and $t({\overline{e}})=o(e)$.
In the present paper, the same symbols $\mathbb X$ and ${\mathbb X}^{\partial}$
also denote the corresponding graphs ${\mathbb G}({\mathbb X}, \rho)$ and
${\mathbb G}^{\partial}({\mathbb X}, \rho)$, respectively.

For $X\in {\mathbb X}$ and $\rho>0$, we set
${\mathcal N}_{\rho}(X):=\bigcup_{Y\in N_{\rho}(X)} Y$. 
Let $V$ be a killing potential. 
For two parameters $\alpha , \beta \geq 0$, 
we define the transition probability $p_{\alpha,\beta }=
p_{\alpha, \beta, {\mathscr X}}$ of a random walk on
the proximity graph $({\mathbb G}^{\partial}({\mathbb X}, \rho), {\mathscr X})$
as follows. For $e=(X,Y)\in {\mathbb E}^{\partial}$ with $X\in {\mathbb X}$ and $Y\in {\mathbb X}^{\partial}$, 
\begin{equation*}
p_{\alpha, \beta }(e)=p_{\alpha, \beta}(X,Y):=\left\{ 
\begin{array}{ll}
{\displaystyle{
\min \{\beta V(x(X)), 1\}
}} & \mbox{if } Y= {X(\partial)}, 
\\
{\displaystyle{
\max \big \{
1-\beta V(x(X)), 0 \big \}
\frac{\mu_\alpha (Y)}{{\mu}_\alpha ({\mathcal N}_{\rho}(X))}
}}
 &  \mbox{if }Y \in N_{\rho}(X), \\
0 & \mbox{otherwise}
\end{array}\right.
\end{equation*}
and for $Y\in {\mathbb X}$,
$$ p_{\alpha, \beta }(X(\partial), X(\partial)):=1, \quad p_{\alpha, \beta }(X(\partial), Y):=0.$$
The transition kernel $p_{\alpha,\beta}$ defines a time-homogeneous Markov chain on
${\mathbb G}^{\partial}({\mathbb X}, \rho)$, which we call the random walk with
{\it{killing rate}} $\beta V$.
We also define a measure $m_\alpha $ on ${\mathbb X}^{\partial}$ by
$m_\alpha (X(\partial))=\mu_\alpha (X(\partial)):=0$ and
$$ m_\alpha (X):=\mu_\alpha (X) \mu_\alpha ({\mathcal N}_{\rho}(X)), \quad X\in {\mathbb X}.$$
When $\beta V=0$, a direct calculation gives
\begin{equation}
 p_{\alpha, \beta }(X,Y)m_\alpha (X)=\mu_\alpha (X)\mu_\alpha (Y)=
p_{\alpha, \beta }(Y,X)m_\alpha (Y), \quad X,Y\in {\mathbb X}^{\partial},
\label{symmetry}
\end{equation}
which means that the corresponding random walk is {\it{$m_\alpha $-symmetric}}.
\subsection{Exclusion process on the proximity graph}
We define the configuration space on the graph ${\mathbb X}^{\partial}$ by 
$${\mathcal X}:=\big \{ \eta=(\eta(X))_{X\in {\mathbb X}^{\partial}};~
\eta(X) \in \{0,1\} \mbox{ for }X\in {\mathbb X} \mbox{ and }\eta(X(\partial))=0 \big\},$$
where $\eta(X)=1$ means the site $X$ is occupied 
by a particle, while $\eta(X)=0$ means the site $X$ is vacant.
For $\eta\in {\mathcal X}$ and $e\in {\mathbb E}$, we denote by $\eta^{e}\in {\mathcal X}$ 
the configuration
defined by exchanging the value of $\eta(o(e))$ and $\eta(t(e))$, i.e., 
\begin{equation*}
\eta^{e}(X)
=\left\{ 
\begin{array}{ll}
\eta({{t(e)}})
&~~X=o(e) \\
\eta({o(e)}) &~~X=t(e)\\
\eta({X}) &~~\mbox{otherwise}.
\end{array}\right. 
\end{equation*} 

Note $\eta^{e}=\eta^{\overline{e}}$ for $e\in {\mathbb E}$.
For $e\in {\mathbb E}^{\partial} \setminus {\mathbb E}$, we also define $\eta^{e}\in {\mathcal X}$ by
\begin{equation*}
\eta^{e}(X)
=\left\{ 
\begin{array}{ll}
0
&~~X=o(e), t(e) \\
\eta({X}) &~~\mbox{otherwise}.
\end{array}\right. 
\end{equation*} 
Throughout the present paper, we usually write $\eta^{Y,Z}$ instead of $\eta^{e}$ for 
$e=(Y,Z)\in  {\mathbb E}^{\partial}$. 

Motivated by the construction of \textit{lamplighter groups} and
\textit{lamplighter graphs} (see, for instance, \cite{Sal01} and \cite{Woe05}),
we regard $\eta^e$ as adjacent to $\eta$. This gives the directed graph
${\mathcal G}({\mathcal X}, \rho)=({\mathcal V}, {\mathcal E})$
on the configuration space ${\mathcal X}$, with ${\mathcal V}:={\mathcal X}$
and
\begin{align*}
{\mathcal E}:= \big \{ {\bm e}=(\eta, \zeta) \in 
{\mathcal X} \times {\mathcal X}; & 
{\mbox{ there exists an }} {e}\in {\mathbb E}^{\partial}
\nonumber \\
&
 {\mbox{ such~that }} \zeta=\eta^{e} \mbox{ with } \eta({o(e)})=1, \eta(t(e))=0
\big \}.
\end{align*}
For $\eta\in {\mathcal X}$, we put
${\mathcal E}_{\eta}=\{ {\bm e}\in {\mathcal E}; o({\bm e})=\eta \}$ and ${\mathcal V}_{\eta}:=\{ t({\bm e}) \in 
{\mathcal X}; {\bm e}\in {\mathcal E}_{\eta} \}$.

We denote by ${\mathcal F}_{c}$ the class of all {\it{local}} functions on $\mathcal X$, that is, the space of all
bounded functions $f=f((\eta(X))_{X\in {\mathbb X}^{\partial}}): {\mathcal X} \to \mathbb R$ 
depending only on finitely many $\eta(X)$'s.
Then, we set ${\bm p}_{\alpha, \beta }({\bm e})
={\bm p}_{\alpha, \beta }(\eta, \eta^{e}):=p_{\alpha, \beta }(e)$ for 
${\bm e}=(\eta, \eta^{e})\in {\mathcal E}$ with $e\in {\mathbb E}^{\partial}$, and define a transition operator
${\mathcal L}_{\alpha, \beta }$ acting on ${\mathcal F}_{c}$ by
\begin{align}
{\mathcal L}_{\alpha, \beta }f(\eta)&:= 
\sum_{{\bm e} \in {\mathcal E}_{\eta}} {\bm p}_{\alpha, \beta }({\bm e}) \big( f(t({\bm e})) -f(\eta) \big), \quad 
\eta \in {\mathcal X}.
\label{transition2}
\end{align}
We may write ${\mathcal L}_{\alpha, \beta }f$ as
\begin{align*}
{\mathcal L}_{\alpha, \beta }f(\eta)&= 
\sum_{\zeta\in {\mathcal V}_{\eta}} {\bm p}_{\alpha, \beta }
(\eta, \zeta) \big( f(\zeta) -f(\eta) \big)
\nonumber \\
&=\sum_{e\in {\mathbb E}^{\partial}} {\bm p}_{\alpha, \beta }(\eta, \eta^{e}) 
\big( f(\eta^{e}) -f(\eta) \big)
{\bf 1}_{\eta(o(e))=1, \eta(t(e))=0} (e) 
\nonumber \\
&= \sum_{e\in {\mathbb E}^{\partial}} {\bm p}_{\alpha, \beta }(\eta, \eta^{e}) 
\eta(o(e))(1-\eta(t(e)))
\big( f(\eta^{e}) -f(\eta) \big)
\nonumber \\
&=
\sum_{X \in \mathbb{X}^\partial }\sum_{Y \sim_\rho X} 
 p_{\alpha, \beta}(X, Y)\eta(X) (1-\eta(Y))\big( f(\eta^{XY}) -f(\eta) \big)
\nonumber \\
&=
\sum_{X \in \mathbb{X}} \eta(X) 
\Big\{
\sum_{Y \sim_\rho X} p_{\alpha, \beta}(X, Y)(1-\eta(Y))\big( f(\eta^{XY}) -f(\eta) \big)
\nonumber \\
& \qquad\qquad\qquad {}+ p_{\alpha, \beta }(X, X(\partial))
\big( f(\eta^{X, X(\partial)}) -f(\eta) \big) \Big\}, \quad \eta \in {\mathcal X},
\end{align*}
In the last line, we used $\eta(X(\partial))=0$. This removes the edges
emanating from the cemetery point $X(\partial)$, while the killing edges
$(X, X(\partial))$, whose origin may be occupied, do survive in the summation.
In Section \ref{exclusion} below, we associate with the transition operator
${\mathcal L}_{\alpha, \beta }$ a continuous-time exclusion dynamics
${\eta^{\alpha, \beta }}=\{{\eta^{\alpha, \beta }}(t)\}_{t\geq 0}
=\{({\eta^{\alpha, \beta }}(t,X))_{X\in {\mathbb X}^{\partial}} \}_{t\geq 0}$
on the configuration space ${\mathcal X}$, constructed for every initial distribution
as a solution of the martingale problem for ${\mathcal L}_{\alpha, \beta }$
acting on ${\mathcal F}_{c}$ (see Proposition \ref{prop existence} and
Remark \ref{rem existence}).
We call it the {\it{exclusion process}} on ${\mathbb G}^{\partial}({\mathbb X}, \rho)$.

\begin{re}\label{rem-sym}
In some literature, the term ``symmetric'' for exclusion processes refers to the condition
${\bm p}_{\alpha, \beta }(\eta, \eta^e) = {\bm p}_{\alpha, \beta }(\eta^e, \eta)$ for any 
$\eta\in\mathcal X$ and $e\in\mathbb E^{\partial}$. See, for instance, \cite[Chapter VIII.2]{Lig85}.
In our case, this condition is not valid in general,
even if the partition $\mathbb X$ is 
an equi-partition with respect to $\mu_\alpha$, that is
$\mu_\alpha(X)=\mu_\alpha(Y)$ for any $X,Y\in \mathbb X$.
However, note that the jump rate $p_{\alpha,\beta}$ is reversible with respect to $m_{\alpha}$ when $\beta V=0$.
To take advantage of this fact in the computations,
we include the weight $m_{1/2}$ in the definition of the empirical measure below.
\end{re}

We introduce a realization map $\pi_{\mathbb X}:\mathcal X \to {\mathcal M}_{+}(M)$ 
by
$$ \pi_{\mathbb X} (\eta):=
\frac{1}{\omega_{d}\rho^{d}} \sum_{X\in {\mathbb X}}\eta({X})
m_{1/2} (X)
\delta_{x(X)}, 
\quad \eta\in {\mathcal X},$$
where 
$\omega_{d}$ is the volume of the unit ball in $\mathbb R^d$, 
$m_{1/2} (X):=\mu_{1/2} (X) \mu_{1/2} ({\mathcal N}_{\rho}(X))$ and
$\delta_{x(X)}$ is the Dirac delta measure at $x(X)\in M$. 
We call $\pi_{\mathbb X}(\eta)$
the {\it{canonical empirical measure}}.
We also introduce analogous 
empirical measures
${\widetilde \pi}_{\mathbb X}(\eta)$ and ${\widehat \pi}_{\mathbb X}(\eta)$ ($\eta\in \mathcal X$) by
\begin{align*}
{\widetilde \pi}_{\mathbb X}(\eta):=\sum_{X\in {\mathbb X}} \eta(X) \mu_{1/2}(X) \delta_{x(X)}
\quad {\mbox{and}} \quad
{\widehat \pi}_{\mathbb X}(\eta):=\sum_{X\in {\mathbb X}} \eta(X)\delta_{x(X)},
\end{align*}
respectively. 
The former approximates the weighted measure $\mu_{1/2}(dx)$ on $M$, whereas the latter preserves the number of particles. 
\begin{re}
At first glance, the definition of the canonical empirical measure $\pi_{\mathbb X}(\eta)$ may appear unusual because
it is weighted by the measure $m_{1/2}$. However, this weighting is natural because it gives the correct
Riemann-sum normalization. Indeed, for a sequence $\{ \mathbb{X}_k\}_{k \in \mathbb{N}}$ such that $| \mathbb{X}_k | \rightarrow 0$ as $k \rightarrow \infty$, there exists a subsequence $k(\rho)$ satisfying $k(\rho) \nearrow \infty$ as $\rho
\searrow 0$ such that for any
$\varphi\in C_c^\infty(M)$, we have
\begin{align*}
\lim_{\rho\searrow0}\frac{1}{\omega_{d}\rho^{d}} \sum_{X\in {\mathbb X}_{k(\rho)}}\varphi(x({X})) m_{1/2} (X)
= \int_M \varphi(x)\mu_{1}(dx),
\end{align*}
under reasonable conditions on $\{ \mathbb{X}_k\}_{k \in \mathbb{N}}$.
In particular, for the fully occupied configuration, omitting the weight $m_{1/2}$
would not produce a Riemann-sum approximation converging to $\mu_1$.
We therefore use $\pi_{\mathbb X}(\eta)$ as the canonical empirical measure and compare it below with the
alternative empirical measures ${\widetilde \pi}_{\mathbb X}(\eta)$ and ${\widehat \pi}_{\mathbb X}(\eta)$.
\end{re}
\subsection{Hydrodynamic equation}
Let us recall that $M=(M,g,\mu_{1})$ is a $d$-dimensional geodesically complete 
weighted Riemannian manifold and $V$ is a non-negative smooth potential function on $M$.
We introduce the notion of a {\it{bounded weak solution}} to
the heat equation with the killing potential term:
\begin{align}\label{hdleq}
\frac{\partial}{\partial t}u(t,x)=\Delta_{1} u(t,x)-V(x)u(t,x),~~u(0,x)=u_{0}(x), \quad t>0, \, x\in M.
\end{align} 

\begin{df}\label{def:weak}
Fix $T>0$ and a Borel measurable function $u_0:M\to[0,1]$. 
We say that $\pi=(\pi_t)_{0\leq t \leq T}\in C([0,T], \mc M_+(M))$
is a bounded weak solution to \eqref{hdleq} 
if the following three conditions hold:
\begin{enumerate}
\item[\rm(i)] $\pi_{0}(dx)=u_{0}(x)\mu_{1}(dx)$.
  \item[\rm(ii)] 
    For every $t\in[0,T]$, the measure $\pi_t$ is absolutely continuous with
    respect to $\mu_1$, and there exists a constant $C>0$ such that its Radon--Nikodym density
    $u(t,x):=(d\pi_{t}/d\mu_1)(x)$ satisfies
    \begin{align}
    0\le u(t,x) \le C \quad \text{for $\mu_1$-a.e. $x\in M$ and every } t\in[0,T].
    \label{RN-unif-bdd}
    \end{align}
  \item[\rm(iii)] 
    For every $\varphi\in C_c^\infty(M)$ and every $t\in[0,T]$,
    \begin{align}
\lan\pi_t, \varphi\ran - \lan \pi_0,\varphi\ran = \int_0^t \lan\pi_s, 
\Delta_{1}\varphi -V\varphi
\ran ds.
\label{weak heat}
\end{align}
\end{enumerate}
\end{df}

\begin{re}
In Definition 2.3, the density $u(t,\cdot)$ is defined for each fixed $t$
separately, and its joint measurability in $(t,x)$ is not a priori clear.
However, a jointly measurable version is obtained canonically as follows.
By the vague continuity of $t \mapsto \pi_t$ and the uniform bound \eqref{RN-unif-bdd},
$\Lambda(F) := \int_0^T \langle \pi_t, F(t,\cdot)\rangle\, dt$, 
$F \in C_c((0,T)\times M)$
defines a positive Radon measure 
on
$(0,T)\times M$ which is bounded by a constant multiple of the product measure
$dt \otimes \mu_{1}$.
Hence the Radon--Nikodym theorem yields a jointly measurable function
$\widetilde{u} \in L^\infty((0,T)\times M)$ with
$0 \le \widetilde{u} \le C$ such that
$\widetilde{u}(t,\cdot) = u(t,\cdot)$ $\mu_1$-a.e.\ for a.e.\ $t \in (0,T)$.
Since the function $u$ enters arguments of Section 4 only through $dt$-integrals,
we may replace $u$ by $\widetilde{u}$ there. Moreover, once the
smoothness of $\widetilde{u}$ is established in Lemma 4.2, the vague
continuity of $t \mapsto \pi_t$ recovers the identity
$\pi_t(dx) = \widetilde{u}(t,x)\mu_1(dx)$ for \emph{every} $t \in (0,T)$.
\end{re}

As in the classical approach, the uniqueness of weak solutions plays
a fundamental role in the proof of the hydrodynamic limit.
In the current non-compact manifold setting, we prove the following result,
which is a natural generalization of the Euclidean setting.

\begin{tm}\label{thm:uniq}
Assume that $(M,g,\mu_1)$ is stochastically complete.  Let $\pi$ and $\tilde{\pi}$ be
bounded weak solutions to \eqref{hdleq} in the sense of Definition \ref{def:weak},
with the same initial measurable function $u_0:M\to[0,1]$. Then
$\pi_{t}={\tilde{\pi}}_{t}$
for all $0\leq t \leq T$.
In particular, the unique bounded solution $\pi$ is given by
$\pi_{t}(dx)=u(t,x)\mu_{1}(dx)$, $0\leq t \leq T$,
where 
$u(0,x)=u_{0}(x)$ and
\begin{align}
u(t,x)
&
=
\int_{M} p_{1}^{V}(t,x,y) u_{0}(y) {\mu}_{1}(dy), \quad 0< t\leq T, \, x\in M.
\label{explicit-representation}
\end{align} 
\end{tm}

The proof of Theorem \ref{thm:uniq} is provided in Section \ref{uni}.

\begin{re}
When $V=0$, if $M$ is not stochastically complete,
bounded classical solutions need not be unique.
Hence, the assumption that $M$ is stochastically complete
is necessary for the validity of Theorem \ref{thm:uniq}.
\end{re}

\subsection{Main results}\label{main results}

Let $\{\mathbb{X}_{k} \}_{k\in {\mathbb N}}$ 
be a sequence of partitions of $M$ satisfying 
conditions {\bf{(A)}}, {\bf{(B)}} and $ |\mathbb{X}_{k}|  \searrow 0$ as $k\to \infty$.
For each $k$, we take an arbitrary set of reference points 
associated with the partition ${\mathbb X}_{k}$
and denote it by ${\mathscr X}_{k}=\{x(X)\}_{X\in {\mathbb X}_{k}}$.
We also denote by ${\mathcal X}_{k}$ the configuration space on the graph ${\mathbb X}_{k}^{\partial}$.

For a probability measure $\nu^{k}$ on ${\mathcal X}_{k}$,
we denote by ${\mathbb P}_{\nu^{k}}$ the probability measure on
$D([0,T], {\mathcal X}_{k})$ with the initial distribution $\nu^{k}$ constructed in
Proposition \ref{prop existence} below, and by ${\eta}^{\alpha, \beta, k}
=\{{\eta}^{\alpha, \beta, k}(t)\}_{t\geq 0}$ the canonical coordinate process under
${\mathbb P}_{\nu^{k}}$, that is, the exclusion process associated with
${\mathcal L}_{\alpha, \beta}$ with the
initial measure $\nu^{k}$ (see Remark \ref{rem existence}(i)). We also introduce
${\eta}^{\alpha, \beta, k}_{\gamma}=\{{\eta}^{\alpha, \beta, k}_{\gamma}(t)\}_{t\geq 0}$ by
$${\eta}^{\alpha, \beta, k}_{\gamma }(t):={\eta}^{\alpha, \beta, k}(\gamma t), \quad t\geq0.$$
Note that the law of ${ \eta}^{\alpha, \beta,  k}_{\gamma}$ solves the martingale
problem associated with
$\gamma {\mathcal L}_{\alpha, \beta}$ (see Remark \ref{rem existence} (ii)).
For simplicity, we put
\begin{equation*}
\delta :=\frac{\rho^2}{2(d+2)}~\mbox{ and }~~
 {\xi}^{\rho}(t):=
{\eta}^{1/2, \delta, k(\rho)}_{1/\delta}(t), \quad t\geq 0.
\end{equation*}

To state our main results, we need to impose the following 
{\it{locally almost equi-partition condition}}
on the choice of $\{ \mathbb{X}_k \}$:
\begin{description}
\item[\bf{(C):}] 
There exists a subsequence $\{ k(\rho) \}_{\rho>0} \subset \mathbb N$ 
satisfying $k(\rho)\nearrow \infty$ and $|\mathbb{X}_{k(\rho)}|=o(\rho^3)$ as $\rho \searrow 0$
such that for a fixed base point $o\in M$ and for all $r>0$
\begin{align} 
\sup \bigg\{ 
\bigg \vert \frac{{\mu}_{1} (Y) }{{\mu}_{1} (X) }-1 \bigg \vert  \, ; \,  &X, Y 
\in \mathbb{X}_{k(\rho)}, X, Y \subset B(o, r),  X\sim_\rho Y
\bigg\}
\nonumber \\
&=o_{B(o,r)}(\rho^2)\quad \mbox{ as }\rho \searrow 0.
\label{hobo toubunkatsu}
\end{align}
\end{description}

We are in a position to state the main result of this paper.
\begin{tm}\label{mainthm}
Let $(M, g, \mu_{1})$ be a geodesically complete weighted Riemannian manifold.
In addition, we assume that $M$ is stochastically complete.
Let $\{ \mathbb{X}_{k}\}$ be a sequence of partitions of $M$
satisfying {\bf (A)}, {\bf (B)} and {\bf (C)}.
Let $u_{0}: M \to [0,1]$ be a Borel measurable function and assume that 
the initial measures $\{ \nu^{k(\rho)} \}_{\rho>0}$ satisfy
\begin{align}
\lim_{\rho \searrow 0} \nu^{k(\rho)}\bigg( \Big \vert 
\big \langle
\pi_{{\mathbb X}_{k(\rho)}}({\xi}^{\rho}(0)),
\varphi
\big \rangle
-\int_{M} u_{0}(x) {\varphi}(x){\mu}_1 (dx) \Big \vert >\varepsilon \bigg)=0
\label{mainthm ini}
\end{align}
for any $\varepsilon>0$ and $\varphi \in C^{\infty}_{c}(M)$. Then 
\begin{align}
\lim_{\rho \searrow 0} {\mathbb P}_{\nu^{k(\rho)}}\bigg( 
\Big \vert 
\big \langle
\pi_{{\mathbb X}_{k(\rho)}}({\xi}^{\rho}(t)),
\varphi
\big \rangle
-\int_{M} u(t,x) \varphi(x) {\mu}_1(dx) \Big \vert 
>\varepsilon \bigg)=0
\label{mainthm claim}
\end{align}
holds for every $t\in [0, T]$, $\varepsilon>0$ and $\varphi\in C^{\infty}_{c}(M)$, 
where $\pi_{t}(dx)=u(t,x) {\mu}_{1}(dx)$ is the unique bounded weak solution
to \eqref{hdleq}.
\end{tm}

\begin{co}\label{Co1}
Under the same assumptions on $M$, $V$ and $\{{\mathbb X}_{k}\}_{k\in \mathbb N}$ as in 
Theorem \ref{mainthm}, assume that
\begin{align}
\lim_{\rho \searrow 0} \nu^{k(\rho)}\bigg( \Big \vert 
\big \langle
\widetilde{\pi}_{{\mathbb X}_{k(\rho)}}({\xi}^{\rho}(0)),
\varphi
\big \rangle
-\int_{M} u_{0}(x) {\varphi}(x){\mu_{1/2}}(dx) \Big \vert >\varepsilon \bigg)=0
\label{Co1 ini}
\end{align}
for any $\varepsilon>0$ and $\varphi \in C^{\infty}_{c}(M)$. Then
\begin{align}
\lim_{\rho \searrow 0} {\mathbb P}_{\nu^{k(\rho)}}\bigg( \Big \vert 
\big \langle
\widetilde{\pi}_{{\mathbb X}_{k(\rho)}}({\xi}^{\rho}(t)),
\varphi
\big \rangle
-\int_{M} u(t,x) {\varphi}(x){\mu_{1/2}}(dx) \Big \vert >\varepsilon \bigg)=0
\label{Co1 claim}
\end{align}
for every $t\in  [0, T]$, $\varepsilon>0$ and $\varphi \in C^{\infty}_{c}(M)$.
\end{co}
\begin{co}\label{Co2}
Under the same assumptions on $M$, $V$ and $\{{\mathbb X}_{k}\}_{k\in \mathbb N}$ as in 
Theorem \ref{mainthm}, assume in addition that 
there exists a base piece $X_{0(k(\rho))} \in \mathbb{X}_{k(\rho)}$ 
such that 
for all $r>0$
\begin{equation}
\sup \left\{ 
 \left| \frac{\mu_{1/2}(X_{0(k(\rho))} )}{\mu_{1/2}(X)}-1 \right| \, ; \, 
 X \in \mathbb{X}_{k(\rho)}, X \subset B(o, r)
 \right\}
  =o_{B(o,r)}(1) \quad \mbox{ as }\rho \searrow 0.
 \label{hatkatei}
\end{equation}
If
\begin{align}
\lim_{\rho \searrow 0} \nu^{k(\rho)}\bigg( \Big \vert \mu_{1/2}(X_{0(k(\rho))})
\big \langle
\hat{\pi}_{{\mathbb X}_{k(\rho)}}({\xi}^{\rho}(0)),
\varphi
\big \rangle
-\int_{M} u_{0}(x) {\varphi}(x)\mu_{1/2}(dx) \Big \vert >\varepsilon \bigg)=0
\label{Co2 ini}
\end{align}
for any $\varepsilon>0$ and $\varphi \in C^{\infty}_{c}(M)$, then
\begin{align}
\lim_{\rho \searrow 0} {\mathbb P}_{\nu^{k(\rho)}} \bigg( \Big \vert 
\mu_{1/2}(X_{0(k(\rho))})
\big \langle
\hat{\pi}_{{\mathbb X}_{k(\rho)}}({\xi}^{\rho}(t)),
\varphi
\big \rangle
-\int_{M} u(t,x) {\varphi}(x){\mu_{1/2}}(dx) \Big \vert >\varepsilon \bigg)=0
\label{Co2 claim}
\end{align}
for every $t\in [0, T]$, $\varepsilon>0$ and $\varphi \in C^{\infty}_{c}(M)$.
\end{co}
%
\begin{exm}[A modification of a square partition of $\mathbb{R}^d$]
\label{ex R^d}
The square partition of $\mathbb{R}^d$ is the simplest example of a partition of a manifold. Let $M=\mathbb{R}^d$ and let the weight function $w$ be $1$. For $k \in  \mathbb{N}$, set
\begin{equation*}
\mathbb{X}_k =\left\{ \frac{1}{k} (I+z) \, ; \, z \in  \mathbb{Z}^d \right\},
\end{equation*}
where $I=[0,1)^d$. This is a typical equi-partition of $\mathbb{R}^d$ 
and it is easy to check that 
$\mathbb{X}_k$ satisfies {\bf{(A)}}, {\bf{(B)}} for all $k \in \mathbb{N}$. 
Moreover, by taking $k(\rho)=[ \rho^{-4} ]$, 
{\bf{(C)}} is satisfied.

For a nontrivial example of $\mathbb{R}^d$, 
let us consider a small modification of $\mathbb{X}_k$. 
For $l>0$, $X = k^{-1}(I+z_X) \in \mathbb{X}_k$, and $X^\prime =X+z_i/k$ for some $z_i=(0 , \ldots,\stackrel{i}{1}, \ldots, 0)$, 
define a modification of $X$ and $X^\prime$ by
\begin{equation*}
Y=X \cup \frac{1}{k} \left( \frac{1}{k^l} I+z_X+z_i \right),  \quad Y^\prime =
X^\prime \setminus \frac{1}{k}\left( \frac{1}{k^l} I+z_X+z_i \right)
\end{equation*}
(see Figure \ref{fig mod}).
We denote by $\mathbb{Y}_k$ the partition of $\mathbb{R}^d$ obtained by applying the above modification
to some adjacent pieces in $\mathbb{X}_k$. Then $\mathbb{Y}_k$ satisfies {\bf{(A)}} and {\bf{(B)}} 
for all $k \in \mathbb{N}$. Moreover,  by taking $k(\rho)=[ \rho^{-4}]$, 
{\bf{(C)}} is satisfied if $l \geq d^{-1}$.
\end{exm}

\begin{exm}\label{rei}
Let $(M, \mu_1)$ be a weighted manifold and suppose that there exists a 
reference point $o \in M$ such that for all $k \in \mathbb{N}$,
$B(o,k)$ has an equi-partition with respect to $\mu_1$ with diameter less than $k^{-1}$.
 Define a 
sequence of partitions $\mathbb{X}_k$ of $M$ as follows: 
In $B(o,k)$, take an equi-partition with diameter less than $k^{-1}$ by the assumption. 
Outside $B(o,k)$, take a partition with diameter less than $k^{-1}$ (for instance, take a Voronoi diagram of a maximal $(4k)^{-1}$-net of $M\backslash B(o,k)$). Then $\mathbb{X}_k$ satisfies {\bf{(A)}} and {\bf{(B)}}. Moreover, by taking 
$k(\rho)=[ \rho^{-4} ]$, $\mathbb{X}_{k}$ satisfies {\bf{(C)}}.
\end{exm}

\begin{re}
{\rm{(i)}}~For example, any model manifold $(\mathbb{R}^d, \psi)$ (see \cite{Gri09} and \cite{IK24}) satisfies the condition of Example \ref{rei}. However, it is not clear that any weighted manifold satisfies this condition. Although equi-partition of $B(o,k)$ always exists, the connectivity and 
the estimate of the diameter of the piece are not trivial. 
This is one reason why we allow partitions that are not equi-partitions.
\vspace{1.5mm} \\
{\rm{(ii)}}~We note that an equi-partition with {\bf{(A)}} and {\bf{(B)}} cannot in general be chosen when the manifold is non-compact. This is another reason why we impose condition {\bf{(C)}}. 
Indeed, when $M$ has a cusp, the diameter of a cell of fixed measure
diverges as it gets closer to the cusp, which violates {\bf{(B)}} (see Figure \ref{fig cusp}). 
\end{re}

\begin{figure}[h]
\centering
\begin{minipage}[b]{0.49\columnwidth}
    \centering\scalebox{0.5}{
\includegraphics{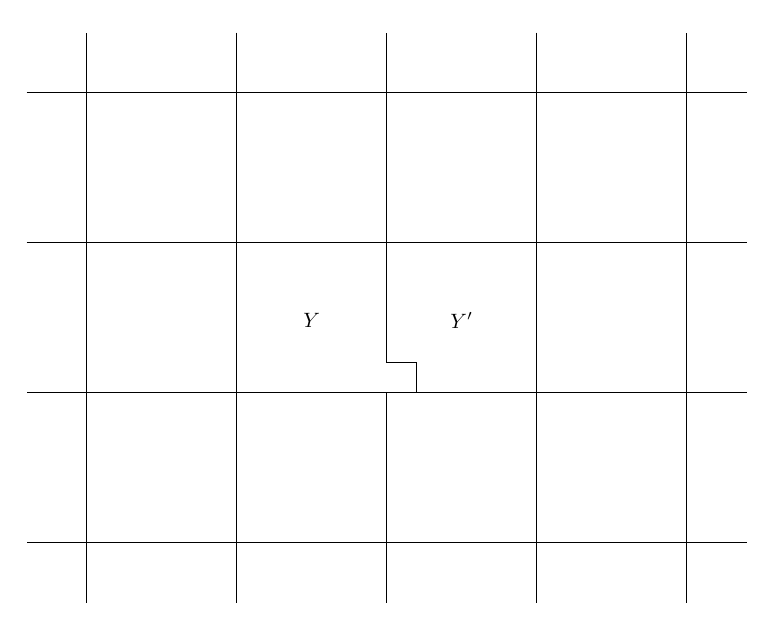}
}
\caption{A modification of a square partition of $\mathbb{R}^2$.}
   \label{fig mod}
\end{minipage}
\begin{minipage}[b]{0.49\columnwidth}
    \centering
\scalebox{0.6}{
\includegraphics{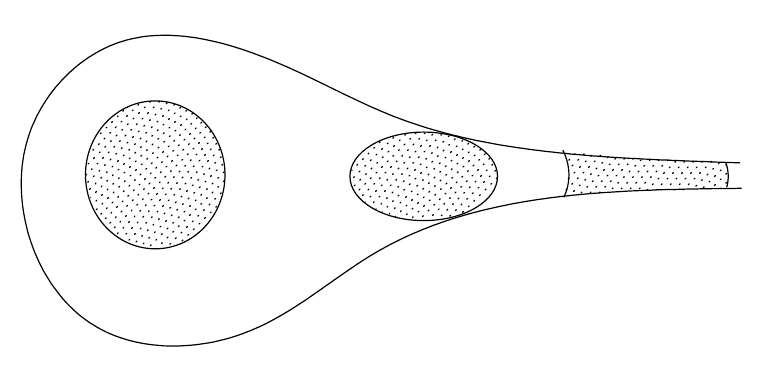}
}
\caption{Pieces of equal measure on a manifold with a cusp}
    \label{fig cusp}
\end{minipage}
\end{figure}
\section{Proof of Theorem \ref{mainthm}}

We prove Theorem \ref{mainthm} in this section.
The proof is based on standard arguments for proving the hydrodynamic limit
for symmetric exclusion processes, cf. \cite[Chapter 4]{KL99} and \cite[Chapter 8]{Sep08}.
We first sketch the proof of Theorem \ref{mainthm}.
Note that $\{\pi_{\bb X_{k(\rho)}}(\xi^\rho(\cdot))\}$
is a $D([0,T], \mc M_+(M))$-valued random variable.
We denote its distribution by ${\mathbb Q}^\rho$, which is a probability measure on $D([0,T], \mc M_+(M))$.
We shall prove that the family $\{{\mathbb Q}^\rho\}_{\rho>0}$ is relatively compact
and every limit point is concentrated on the set of weak solutions to \eqref{hdleq}. The precise statements are given in Lemmas \ref{tight} and \ref{cha}, respectively.
The proof of Theorem \ref{mainthm} is completed by invoking the uniqueness
result for the hydrodynamic equation \eqref{hdleq},
which will be established in Theorem \ref{thm:uniq}.
Because the manifold is non-compact, the required uniqueness result is proved separately in Section \ref{uni}.

\subsection{Preliminary computations}

In this subsection, we provide several estimates that appear in the following subsections.
We begin with simple but repeatedly used estimates.

\begin{lm}\label{Lem nbd}
Let $X \in \mathbb{X}$ be a piece and $\rho>0$ so that the exponential map at $x(X)$ 
$\exp_{x(X)}: T_{x(X)} M \rightarrow M$ is diffeomorphic on 
$B^{\mathbb{R}^d}(\rho +|\mathbb{X}|) \subset T_{x(X)} M$. Then
the following estimates hold:
\begin{align}
&\frac{\mu_{1/2} (\mathcal{N}_\rho (X))}{\omega_d \rho^d}
=w^{1/2} (x(X)) 
+\left( \Delta \left( w^{1/2} \right) (x(X))-\frac{w^{1/2}(x(X)) 
\mathrm{Scal}_g(x(X)) }{3} \right)
\frac{\rho^2}{2(d+2)} 
\nonumber \\
&\mbox{ }\hspace{25mm} +O_{B(x(X), \rho+|\mathbb{X}|) } \Big( \frac{|\mathbb{X}|}{\rho} + \rho^3 \Big) ,
\label{nbd} \\
&w^{1/2} (x(X))\mu_{1/2} (X)=\big(1+ O_{B(x(X), |\mathbb{X}|)}( |\mathbb{X}|) \big) \mu_1(X),
\label{1/2 to 1}
\end{align}
where $\mathrm{Scal}_g$ is the
scalar curvature of $(M,g)$. 
\end{lm}
\begin{proof}
First, we prove (\ref{nbd}). By the definition of $\mathcal{N}_\rho (X)$, 
$$
B(x(X), \rho-|\mathbb{X}|) \subset \mathcal{N}_\rho (X) \subset B(x(X), \rho+|\mathbb{X}|).
$$ 
This implies that there exists $-1\leq  \varepsilon \leq 1$ such that 
\begin{equation*}
\mu_{1/2} (\mathcal{N}_\rho (X)) =\mu_{1/2} \left( B(x(X), \rho+\varepsilon |\mathbb{X}| )\right).
\end{equation*}
Using the normal coordinates $y=(y^1, \ldots , y^d)$ centered at $x(X)$, 
by the assumption of the exponential map, we obtain
\begin{align*}
&\mu_{1/2} \left( B(x(X), \rho+\varepsilon |\mathbb{X}| )\right)=
\int_{B_{\mathbb{R}^d}( \rho+\varepsilon |\mathbb{X} | )} 
w^{1/2}(y)\sqrt{ \det g (y) } dy.
\end{align*}
Applying Taylor's expansion formula, we obtain
\begin{align*}
\mu_{1/2} &
\big( B(x(X), \rho+\varepsilon |\mathbb{X}| )\big)\\
&=
\int_{B_{\mathbb{R}^d}( \rho+\varepsilon |\mathbb{X} | )} 
\left(
w^{1/2}(0)+
\partial_i \left(w^{1/2}\right) (0) y^i+
\frac{1}{2}\partial_i\partial_j \left(w^{1/2}\right)(0) y^i y^j +O_{B(x(X), \rho+|\mathbb{X}|) }(|y|^3) 
 \right)
 \nonumber \\
 &
 \hspace{30mm}
 \times
 \left(1-\frac{1}{6} \mathrm{Ric}_{ij}(0) y^i y^j +O_{B(x(X), \rho+|\mathbb{X}|)}(|y|^3) \right) dy\\
 &=w^{1/2}(x(X)) \omega_d (\rho +\varepsilon |X| )^d \\
&\quad
+\left( 
 \frac{1}{2} \Delta (w^{1/2}) (x(X)) - \frac{w^{1/2}(x(X))}{6} 
 \mathrm{Scal}_g(x(X)) 
 \right)\frac{ \omega_d  (\rho+\varepsilon |\mathbb{X} | )^{d+2}}{d+2} \\
 &\quad
 +O_{B(x(X), \rho+|\mathbb{X}|)}( \rho^{d+3}),
\end{align*} 
which concludes (\ref{nbd}). Here $\mathrm{Ric}_{ij}(0)$
 is the $(i,j)$-component of the Ricci curvature tensor at $x(X)$.

Next, we prove (\ref{1/2 to 1}). Applying Taylor's expansion formula 
on the normal coordinates centered at $x(X)$, 
there exists $y^\prime \in B^{\mathbb{R}^d}(|\mathbb{X}|) \subset 
T_{x(X)}M$ such that 
\begin{align*}
\mu_{1/2}(X)&= \int_{\exp^{-1}_{x(X)}(X)}  w^{1/2}(y) \sqrt{ \det g (y)} dy \\
&= \int_{\exp^{-1}_{x(X)}(X)} w^{-1/2}(y) w(y) \sqrt{ \det g (y) } dy\\
&= \int_{\exp^{-1}_{x(X)}(X)}  \left( w^{-1/2}(0) +\partial_i (w^{-1/2}) (y^\prime) y^i 
\right) w(y) \sqrt{ \det g (y) } dy.
\end{align*}
Then there exists a constant $C=C(X)$ such that  
\begin{align*}
&\left\vert w^{1/2}(x(X)) \mu_{1/2}(X)-\mu_1(X) \right\vert
\nonumber \\
&~~~
\leq  
Cw^{1/2}(x(X)) \| \nabla (w^{-1/2} ) \|_{L^\infty (B(x(X), \mathrm{diam}(X))}
\mathrm{diam} (X) 
\mu_1 (X),
\end{align*}
which concludes (\ref{1/2 to 1}).
\end{proof}

We next show that
\eqref{hobo toubunkatsu} in condition {\bf (C)}
still holds by replacing $\mu_1$ with $m_{1/2}$.

\begin{lm}
\label{hobo 1/2}
 If $\mathbb{X}_{k}$ satisfies condition {\bf (C)}, then
 for every $r>0$, 
\begin{equation*}
\sup\left\{ \left|  \frac{m_{1/2}(Z)}{m_{1/2}(Y)} -1 \right| \, ;\,
Y, Z \in \mathbb{X}_{k(\rho)}, Y,Z \subset B(o, r), Y\sim_{\rho} Z,\right\} =
o_{B(o,r+2)}(\rho^2)
\quad \mbox{ as }\rho \searrow 0.
\end{equation*}
\end{lm}
\begin{proof}
For $X \in \mathbb{X}_{k(\rho)}$, set
$$
\mu_{1/2}^{w}(X)=
\frac{\mu_{1/2}(\mathcal{N}_\rho (X) )}{ w^{1/2}(x(X)}.
$$
Then we have for $Y, Z \in \mathbb{X}_{k(\rho)}$ in concern
\begin{align*}
\left\vert \frac{m_{1/2}(Z)}{m_{1/2}(Y)}-1 \right\vert =&
\left\vert \frac{w^{1/2}(x(Z)) \mu_{1/2}(Z) \mu_{1/2}^{w}(Z)}
{w^{1/2}(x(Y)) \mu_{1/2}(Y)  \mu_{1/2}^{w}(Y)}  -1
\right\vert\\
\leq &\frac{\mu_{1/2}^{w}(Z)}{\mu_{1/2}^{w}(Y)}
\left\vert \frac{w^{1/2}(x(Z)) \mu_{1/2}(Z) }
{w^{1/2}(x(Y)) \mu_{1/2}(Y) }  -1
\right\vert
+
\left\vert \frac{\mu_{1/2}^{w}(Z)}{\mu_{1/2}^{w}(Y)} -1
\right\vert\\
=& \mathrm{(I)}+\mathrm{(II)}.
\end{align*}
We first estimate (II).
By Lemma \ref{Lem nbd}, we obtain
\begin{align*}\mu_{1/2}^w (X)&= \omega_d \rho^d \left( 1+ F(x(X)) \rho^2
+O_{B(x(X), \rho+|\mathbb{X}_{k(\rho)}|)}\left(\frac{|\mathbb{X}_{k(\rho)} |}{\rho} + \rho^3\right) \right),
\end{align*}
where
$$F(x)=\frac{1}{2(d+2)} \left( \frac{\Delta \left( w^{1/2} \right) (x)}{w^{1/2}(x)}
 -\frac{\mathrm{Scal}_{g}(x) }{3} \right).$$
Since $B(x(Y), \rho+|\mathbb{X}_{k(\rho)}|) , B(x(Z), \rho+|\mathbb{X}_{k(\rho)}|) \subset
B(o, r+\rho +|\mathbb{X}_{k(\rho)}|)$, we obtain
\begin{align*}
\mathrm{(II)}=&\left\vert \frac{ \mu_{1/2}^w (Z)  -\mu_{1/2}^w (Y)}{\mu_{1/2}^w (Y)}
\right\vert \\
=& \left\vert \frac{
\left(F(x(Z))-F(x(Y)) \right)\rho^2+ O_{B(o, r+\rho+|\mathbb{X}_{k(\rho)}|)}
\left(\frac{|\mathbb{X}_{k(\rho)} |}{\rho} + \rho^3\right)
}
{1+F(x(Y))\rho^2+O_{B(x(Y), \rho+|\mathbb{X}_{k(\rho)}|)}
\left(\frac{|\mathbb{X}_{k(\rho)} |}{\rho} + \rho^3\right)}
\right\vert \\
\leq &
 \frac{
 \| \nabla F\|_{L^\infty(B(x(Y), \rho+|\mathbb{X}_{k(\rho)}|))}\rho^3+
 O_{B(o, r+\rho+|\mathbb{X}_{k(\rho)}|)}\left(\frac{|\mathbb{X}_{k(\rho)} |}{\rho} + \rho^3\right)
}
{1+F(x(Y))\rho^2+O_{B(x(Y), \rho+|\mathbb{X}_{k(\rho)}|)}\left(\frac{|\mathbb{X}_{k(\rho)} |}{\rho} + \rho^3\right)}.
\end{align*}
Taking $\rho>0$ small enough,
the assumption that $|\mathbb{X}_{k(\rho)}|=o(\rho^3)$ implies $\mathrm{(II)}=
o_{B(o, r+2)}(\rho^2)$.

Next, we estimate (I). By the estimate of (II), it suffices to prove that
$$
\left\vert \frac{w^{1/2}(x(Z)) \mu_{1/2}(Z) }
{w^{1/2}(x(Y)) \mu_{1/2}(Y) }  -1
\right\vert =o_{B(o, r+2)}(\rho^2).
$$
Indeed, by using (\ref{1/2 to 1}),
\begin{align*}
\left\vert \frac{w^{1/2}(x(Z)) \mu_{1/2}(Z) }
{w^{1/2}(x(Y)) \mu_{1/2}(Y) }  -1
\right\vert  =&
\left\vert \frac{w^{1/2}(x(Z)) \mu_{1/2}(Z) - w^{1/2}(x(Y)) \mu_{1/2}(Y)}
{w^{1/2}(x(Y)) \mu_{1/2}(Y) }  
\right\vert \\
=&
\left\vert
\frac{(1+O_{ B(x(Z), |\mathbb{X}_{k(\rho)}|) } (|\mathbb{X}_{k(\rho)}| ))\mu_1(Z) -(1+O_{B(x(Y), |\mathbb{X}_{k(\rho)}| )})\mu_1(Y) }{(1+O_{B(x(Y), |\mathbb{X}_{k(\rho)}| ) }(|\mathbb{X}_{k(\rho)}|) )\mu_1(Y) } \right\vert \\
\leq &~
\frac{(1+O_{ B(x(Z), |\mathbb{X}_{k(\rho)}|) } (|\mathbb{X}_{k(\rho)}| ))}
{(1+O_{B(x(Y), |\mathbb{X}_{k(\rho)}| ) }(|\mathbb{X}_{k(\rho)}|) )}
\left\vert \frac{\mu_1(Z)}{\mu_1(Y)} -1 \right\vert \\
& +\left\vert
\frac{\left( O_{B(x(Y), |\mathbb{X}_{k(\rho)}| )}(|\mathbb{X}_{k(\rho)}|)+
O_{B(x(Z), |\mathbb{X}_{k(\rho)}| )}(|\mathbb{X}_{k(\rho)}|) \right) }{(1+O_{B(x(Y), |\mathbb{X}_{k(\rho)}| ) }(|\mathbb{X}_{k(\rho)}|) ) } \right\vert.
\end{align*}
Since $B(x(Y), |\mathbb{X}_{k(\rho)}|), B(x(Z), |\mathbb{X}_{k(\rho)}|) \subset  B(o,r+|\mathbb{X}_{k(\rho)}|)$,
taking $\rho>0$ small enough, condition {\bf{(C)}} implies that  
$\mathrm{(I)}=o_{B(o, r+2)}(\rho^2)$. Combining estimates (I) and (II),
the proof of the lemma is completed.
\end{proof}

The following lemma gives the approximation of the discrete generator by the continuum Laplacian $\Delta_1$ as $\rho\searrow0$.

\begin{lm}[cf. \mbox{\cite[Theorem 3.2]{IK24}}]
\label{one particle}
Let $\varphi \in C^{\infty}_{c}(M)$, and 
take $\rho>0$ small enough so that 
\begin{equation} \rho <\frac{1}{2}\min \left\{ \mathrm{inj}_M (U_2(\supp \varphi)), 1 \right\},
\label{inj}
\end{equation}
where $\mathrm{inj}_M (U_2(\supp \varphi))$
is the infimum of the injectivity radius in $U_2(\supp \varphi)$.
Then for any partition ${\mathbb X}$ satisfying
$| \mathbb{X} | <\rho/3$, there exist 
positive constants 
${C}_{1}$ and ${C}_{2}$
depending 
on $\|\varphi \|_{C^3}$ and $U_2 (\supp \varphi)$ 
such that for all $X \in \mathbb{X}$
\begin{equation*}
\left| \frac{1}{\delta} \left( \sum_{d_H(X,Y)<\rho} p_{1/2, 0}(X, Y)  
\varphi (x(Y)) - \varphi (x(X)) \right) - \Delta_{1} \varphi  (x(X)) \right| \leq C_1\frac{|\mathbb{X}|}{\rho^2} +C_2 \rho.\end{equation*}
\end{lm}

\begin{proof}
It suffices to prove the lemma when $X \subset U_2(\supp \varphi)$. 
First, the argument in the proof of \cite[Theorem 3.2]{IK24} gives
\begin{align}
\left| \frac{1}{\delta} \left( \sum_{d_H(X,Y)<\rho} p_{1/2, 0}(X, Y)  
\varphi (x(Y)) - \varphi (x(X))\right)
-\psi(x(X))
\right| 
\leq C\frac{|\mathbb{X} | }{\rho^2},
\label{weight1}
\end{align}
where
\begin{align*}
\psi(x(X)):=\frac{1}{\delta \mu_{1/2} ( \mathcal{N}_\rho (X)) }
\int_{B(x(X), \rho)}
\left( \varphi (y)- \varphi (x(X)) \right) w^{1/2} (y) d\mathrm{vol}(y)
\end{align*}
and $X=X(x)$, which is the unique piece in $\mathbb{X}$ containing $x$.
By the assumption of $\rho>0$,
 applying Taylor's expansion formula in the normal coordinates 
at $x(X)$, we obtain
\begin{align*}
&\psi(x(X))
\nonumber \\
&=\frac{1}{\delta \mu_{1/2}(\mathcal{N}_\rho (X) )}
\int_{B^{\mathbb{R}^d}(\rho)}
\left( \partial_i \varphi(0) y^i +\frac{1}{2} \partial_i \partial_j \varphi (0) y^iy^j +O_{B(x(X), \rho+|\mathbb{X}|)}(y^3) 
\right) \\
&\quad\times 
\left( w^{1/2} (0)+\partial_i (w^{1/2}) (0)y^i 
+O(y^2) \right)  
\left( 1-\frac{1}{6} \mathrm{Ric}_{ij}(0) y^iy^j +O_{B(x(X), \rho+|\mathbb{X}|)}(y^3) \right) dy \\
&=\frac{1}{\delta \mu_{1/2} ( \mathcal{N}_\rho (X)) }
\left\{ 
\partial_i \varphi (0) w^{1/2}(0) 
 \int_{B^{\mathbb{R}^d}(\rho)} y^i dy \right.\\
& \left.  \quad+\left( \partial_i \varphi (0) \partial_j (w^{1/2}) (0) +\frac{1}{2} 
 w^{1/2}(0) \partial_i \partial_j \varphi (0) \right) 
 \int_{B^{\mathbb{R}^d}(\rho) }y^iy^j dy +O_{B(x(X), \rho+|\mathbb{X}|)}(\rho^{d+3})
 \right\}.
\end{align*}
Using the estimate of $\mu_{1/2}(\mathcal{N}_\rho (X))$ in (\ref{nbd}), 
taking $\rho>0$ 
small enough, we obtain
\begin{align}
\left| \psi(x(X))
- \frac{1}{w^{1/2}(x(X))} g_{x(X)} (\nabla \varphi, \nabla (w^{1/2}) ) -\Delta_{1/2} \varphi (x(X))
\right| =O_{B(x(X), 2)}(\rho).
\label{weight2}
\end{align}
Since
\begin{equation*}
\Delta_1 \varphi (x)= \frac{1}{w^{1/2}(x)} g_x (\nabla \varphi, \nabla (w^{1/2}) ) + \Delta_{1/2} \varphi (x),
\end{equation*}
combining the estimates in (\ref{weight1}) and (\ref{weight2}), 
we conclude the proof.
\end{proof}
\subsection{Existence of the exclusion process} 
\label{exclusion}

In this subsection, we fix $\rho>0$, a partition ${\mathbb X}$ of $M$ satisfying
{\bf (A)} and {\bf (B)} with $\vert {\mathbb X} \vert<\rho/3$, and $\alpha, \beta \geq 0$,
and construct the exclusion process associated with the transition operator
${\mathcal L}_{\alpha, \beta}$ given in \eqref{transition2}.
If the number of neighboring pieces were uniformly bounded, that is,
$\sup_{X \in {\mathbb X}} \sharp N_{\rho}(X)<\infty$, then the total jump rates
influencing each fixed piece would be uniformly bounded, and the standard existence
theorem for interacting particle systems \cite[Theorem I.3.9]{Lig85} would apply
directly and would produce a Feller Markov process.
Such a uniform bound holds, for instance, if the discretization is constructed from a
uniformly separated net under a uniform lower bound on the Ricci curvature, thanks to
the classical packing argument of Kanai \cite{Kan85}.
In the present paper, however, we impose neither a curvature bound nor any global
uniformity on the partition, and condition {\bf (A)} guarantees only the local
finiteness $\sharp N_{\rho}(X)<\infty$ for each fixed $X \in {\mathbb X}$.
For this reason, we characterize the dynamics as a solution of the martingale problem
for ${\mathcal L}_{\alpha, \beta}$ acting on the class ${\mathcal F}_{c}$ of bounded
local functions as in the following proposition.

\begin{pr}\label{prop existence}
For every $T>0$ and every probability measure $\nu$ on $\mathcal X$, there
exists a probability measure $\mathbb P_\nu$ on
$D([0,T],\mathcal X)$ under which the canonical coordinate process
$\eta=\{\eta(t)\}_{0\leq t\leq T}$ has initial distribution $\nu$ and
solves the martingale problem for $(\mathcal L_{\alpha, \beta}, \mathcal F_c)$.
Namely, for every $f\in\mathcal F_c$,
\begin{equation}
M_t^f
:=
f(\eta(t))-f(\eta(0))
-\int_0^t
(\mathcal L_{\alpha,\beta}f)(\eta(s))\,ds,
\qquad 0\leq t\leq T,
\label{mart prob}
\end{equation}
is a square-integrable martingale under $\mathbb P_\nu$
with respect to the canonical filtration $(\mathscr F_t)_{0\leq t\leq T}$,
and its predictable quadratic variation is given by
\begin{equation}
\langle M^f\rangle_t
=
\int_0^t
\bigl(
\mathcal L_{\alpha,\beta}(f^2)
-2f\mathcal L_{\alpha,\beta}f
\bigr)(\eta(s))\,ds,
\qquad 0\leq t\leq T.
\label{mart prob qv}
\end{equation}
\end{pr}

\begin{proof}
The proof of this proposition is standard and follows from
\cite[Chapter 4, Theorem 5.4 and Remark 5.5]{EK86}. 
For this purpose, we confirm the validity of the assumptions in \cite[Chapter 4, Theorem 5.4]{EK86}.
Note that ${\mathcal X}$ endowed with the product topology
is a compact metrizable space.
Therefore, the Stone--Weierstrass theorem shows that ${\mathcal F}_{c}$ is dense in $C({\mathcal X})$.
Note also that $\mathcal L_{\alpha,\beta}f$ is a local function for any $f\in\mathcal F_c$.
In particular, we have $\mathcal L_{\alpha,\beta}\mathcal F_c \subset C(\mathcal X)$.
Moreover, $\mathcal L_{\alpha,\beta}$ satisfies the positive maximum principle. Indeed, suppose that
$f\in{\mathcal F}_{c}$ and that
$f(\eta_{0})=\sup_{\eta\in{\mathcal X}}f(\eta)\geq0$. 
Since $f(\zeta)-f(\eta_{0})\leq0$ for every configuration $\zeta\in\mathcal X$,
we have
$({\mathcal L}_{\alpha,\beta}f)(\eta_{0})\leq0$.
It is also clear that ${\mathcal L}_{\alpha,\beta}{\bf 1}=0$.
Therefore, the cited theorem yields, for every probability measure $\nu$ on
$\mathcal X$, a probability measure $\mathbb P_\nu$ on
$D([0,T],\mathcal X)$ with initial distribution $\nu$ satisfying
\eqref{mart prob} for every $f\in\mathcal F_c$.

It remains to show \eqref{mart prob qv} for any $f\in\mathcal F_c$.
Applying the martingale problem to $f$ and $f^{2}$ shows that
\[
 (M^{f}_{t})^{2}
 -\int_{0}^{t}\bigl(
 {\mathcal L}_{\alpha,\beta}(f^{2})
 -2f{\mathcal L}_{\alpha,\beta}f
 \bigr)(\eta(s))\,ds
\]
is a martingale. Furthermore, a direct computation gives
\[
 \bigl({\mathcal L}_{\alpha,\beta}(f^{2})
 -2f{\mathcal L}_{\alpha,\beta}f\bigr)(\eta)
 =
 \sum_{{\bm e}\in{\mathcal E}_{\eta}}
 {\bm p}_{\alpha,\beta}({\bm e})
 \bigl(f(t({\bm e}))-f(\eta)\bigr)^{2}\geq0.
\]
Therefore, the right-hand side in \eqref{mart prob qv} is continuous, predictable, adapted, and non-decreasing.
The uniqueness of the Doob--Meyer decomposition implies \eqref{mart prob qv}.
\end{proof}

\begin{re}\label{rem existence}
{\rm{(i)}} Throughout the rest of the present paper, for each initial distribution
$\nu$ we fix one probability measure ${\mathbb P}_{\nu}$ given by Proposition
\ref{prop existence}, and the exclusion process $\eta^{\alpha, \beta}$ on
${\mathbb G}^{\partial}({\mathbb X}, \rho)$ means the canonical coordinate process
under ${\mathbb P}_{\nu}$. We do not address the uniqueness of solutions to the above
martingale problem. This causes no ambiguity in our main results because the proof of
Theorem \ref{mainthm} uses only the martingale properties \eqref{mart prob} and
\eqref{mart prob qv} together with deterministic estimates on
${\mathcal L}_{\alpha, \beta}$, and the uniqueness of the limit is guaranteed at the
macroscopic level by Theorem \ref{thm:uniq}. In particular, the conclusion of
Theorem \ref{mainthm} is valid for {\it{every}} choice of the family
$\{ {\mathbb P}_{\nu^{k(\rho)}} \}_{\rho>0}$ of martingale solutions.
\vspace{2mm} \\
{\rm{(ii)}} For $\gamma>0$, define the time scaling
$\Theta_{\gamma}: D([0,\gamma T], {\mathcal X}) \to D([0,T], {\mathcal X})$ by
$(\Theta_{\gamma}\omega)(t):=\omega(\gamma t)$. Applying Proposition
\ref{prop existence} with the time horizon $\gamma T$ and performing a change of
variables in \eqref{mart prob} and \eqref{mart prob qv}, we see that under the image
measure ${\mathbb P}_{\nu}\circ \Theta_{\gamma}^{-1}$ the canonical process solves the
martingale problem associated with $\gamma {\mathcal L}_{\alpha, \beta}$. More precisely, for every
$f \in {\mathcal F}_{c}$,
$f(\eta(t))-f(\eta(0))-\gamma \int_{0}^{t}({\mathcal L}_{\alpha, \beta}f)(\eta(s))\, ds$
is a square-integrable martingale with the predictable quadratic variation
$\gamma \int_{0}^{t}({\mathcal L}_{\alpha, \beta}(f^{2})
-2f{\mathcal L}_{\alpha, \beta}f)(\eta(s))\, ds$. This applies to the time-changed
process $\eta^{\alpha, \beta, k}_{\gamma}$ introduced in Section \ref{main results},
and in particular to the diffusively rescaled process $\xi^{\rho}$.
\vspace{2mm} \\
{\rm{(iii)}} If, in addition, $\sup_{X\in {\mathbb X}}\sharp N_{\rho}(X)<\infty$
holds, then the jump rates satisfy the uniform summability condition required in
\cite[Theorem I.3.9]{Lig85}, and Liggett's theorem provides a Feller Markov process
whose generator is the closure of ${\mathcal L}_{\alpha, \beta}\vert_{{\mathcal F}_{c}}$,
with ${\mathcal F}_{c}$ being a core for the generator. In this case, the above
martingale problem is well-posed, and its unique solution is the path law of the
Feller process given by Liggett's theorem, cf.\ \cite[Theorem 4.4.1]{EK86}.
\end{re}
\subsection{Dynkin's martingale}
In this subsection, we compute Dynkin's martingales,
which play a fundamental role in the proof of Theorem \ref{mainthm}.
First of all, we fix a test function $\varphi\in C^{\infty}_{c}(M)$, and consider
a local function $f_{\rho}=f_{\rho}(\cdot; {\varphi})$
on ${\mathcal X}_{k(\rho)}$ 
defined by
\begin{align}
f_{\rho}(\eta)=
\langle \pi_{{\mathbb X}_{k(\rho)}} (\eta), \varphi \rangle
:=\frac{1}{\omega_{d}\rho^{d}}\sum_{X\in {\mathbb X}_{k(\rho)}} \eta(X) m_{1/2}(X) \varphi(x(X)), \quad \eta\in\mc X_{k(\rho)}.
\end{align}
Note that the right-hand side is a finite sum because $\varphi$ has compact support.

We then define a continuous-time real-valued stochastic 
process $M^{\rho}(\varphi)=\big(M^{\rho}(\varphi)_{t}\big)_{t\geq 0}$ by
\begin{align}\label{dynkin}
M^{\rho}(\varphi)_{t}&:=
f_{\rho}(\xi^{\rho}(t); \varphi)
-
f_{\rho}(\xi^{\rho}(0); \varphi)
-\frac{1}{\delta}
\int_{0}^{t} \big({\mathcal L}_{1/2, \delta}
f_{\rho}(\cdot; \varphi) \big)
(\xi^{\rho}(s))
ds, \quad t\geq 0.
\end{align}
Since $f_{\rho}(\cdot; \varphi)$ and $f_{\rho}(\cdot; \varphi)^{2}$ are bounded
local functions, applying Proposition \ref{prop existence} and
Remark \ref{rem existence}(ii) to the law of the time-changed process $\xi^{\rho}$,
we see that $M^{\rho}(\varphi)$ is a square-integrable martingale with the
quadratic variation $\big( \langle
M^{\rho}(\varphi) \rangle_{t} \big)_{t\geq 0}$ satisfying
$$
\langle M^{\rho}(\varphi) \rangle_{t}=
\int_0^t \Gamma^{\rho}({\xi}^{\rho}(s);\varphi) ds,
$$
where $\Gamma^{\rho}(\eta; \varphi)$ is the so-called {\it{carr\'e du champ}} defined by
\begin{align}\label{carreduchamp}
\Gamma^{\rho}(\eta; \varphi):=
\frac{1}{\delta} \Big \{ {\mathcal L}_{1/2, \delta } \big(f_{\rho}(\cdot; \varphi)^{2} \big)
(\eta)
-2\big({\mathcal L}_{1/2, \delta } f_{\rho}(\cdot; \varphi) \big)(\eta)\cdot
f_{\rho}(\eta; \varphi)
\Big \}.
\end{align}

In the next lemma, we compute Dynkin's martingale $M^\rho(\varphi)$
and express it as a function of the empirical measure $\pi_{{\mathbb X}_{k(\rho)}}$.

\begin{lm}\label{lm1}
For any $\varphi\in C_c^\infty(M)$ and $\eta \in \mathcal{X}_{k(\rho)}$, we have
\begin{equation}\label{generator}
\frac{1}{\delta} 
\left( \mathcal{L}_{1/2, \delta} f_\rho ( \cdot ; \varphi ) \right) (\eta)
= \langle \pi_{\mathbb{X}_{k(\rho)}} (\eta), \Delta_1 \varphi -V\varphi \rangle +o(1) 
\quad \mbox{as }\rho \searrow 0.
\end{equation}
Consequently, for any $0\le t \le T$, we have
\begin{align*}
M^\rho
(\varphi)_{t}&=
\big \langle
\pi_{{\mathbb X}_{k(\rho)}}
\big( \xi^\rho(t)\big), \varphi
\big \rangle
-
\big \langle
\pi_{{\mathbb X}_{k(\rho)}}
\big( \xi^\rho(0)\big), \varphi
\big \rangle
\nonumber \\
&\qquad
-\int_{0}^{t} 
\langle
\pi_{{\mathbb X}_{k(\rho)}}
\big(\xi^\rho(s)\big),
\Delta_1 \varphi - V\varphi
\big \rangle ds +o(1).
\end{align*}
\end{lm}

\begin{proof} 
Fix $\varphi\in C_c^\infty(M)$ and $\rho >0$ so that 
\begin{equation*}
\delta=\frac{\rho^2}{2(d+2)} \leq \min \left\{ \| V \|^{-1}_{U_{2}(\supp \varphi)}, 1 \right\},
\end{equation*}
and (\ref{inj}) holds. 
By the definition (\ref{transition2}) 
of $\mathcal{L}_{1/2, \delta}$, we obtain for any $\eta \in \mathcal{X}_{k(\rho)}$
\begin{align}
&
\frac{1}{\delta \omega_d \rho^d }\sum_{Y\in {\mathbb X}_{k(\rho)}^{\partial}}
 \sum_{Z
\in ({\mathbb E}_{k(\rho)}^{\partial})_{Y}}
p_{1/2, \delta} (Y,Z)\eta(Y)\big(1-\eta(Z) \big) 
\sum_{X\in {\mathbb X}^{\partial}}\big(\eta^{Y,Z}(X)-\eta(X) \big)\varphi(x(X)) m_{1/2}(X)
\nonumber \\
&=
\frac{1}{\delta \omega_{d}\rho^d }\sum_{Y\in {\mathbb X}_{k(\rho)}^{\partial}}
\eta(Y) \bigg\{
\sum_{Z \in ({\mathbb E}_{k(\rho)})_{Y}}
p_{1/2, \delta} (Y,Z)\big(1-\eta(Z) \big) 
\big(\varphi(x(Z))m_{1/2} (Z)-\varphi(x(Y))m_{1/2} (Y) \big) 
\nonumber \\
& 
\hspace{30mm}
+p_{1/2, \delta} (Y, X(\partial)) \left( \varphi(x(\partial)) m_{1/2}(X(\partial))
- \varphi (x(Y)) m_{1/2}(Y) \right) \bigg\} 
\nonumber \\
&=
\frac{1}{\delta \omega_d \rho^d }
\sum_{Y\in {\mathbb X}_{k(\rho)}^{\partial}}
\eta(Y) \bigg\{ 
\left( 1 -\delta V(x(Y)) \right) 
\sum_{Z \in ({\mathbb E}_{k(\rho)})_{Y}} \frac{\mu_{1/2}(Z)}{\mu_{1/2}(\mathcal{N}_\rho(Y) )}
\big(1-\eta(Z) \big)
\nonumber \\
&\mbox{ }
\hspace{30mm}
\times
\big(\varphi(x(Z))m_{1/2} (Z)-\varphi(x(Y))m_{1/2} (Y) \big) 
-\delta V(x(Y)) \varphi (x(Y)) m_{1/2}(Y)  \bigg\} 
\nonumber \\
&=
\frac{1}{\delta \omega_d \rho^d }
\sum_{Y \in \mathbb{X}_{k(\rho)}} \eta(Y) 
\sum_{Z \in \mathbb{E}_{k(\rho) Y}} 
\frac{\mu_{1/2}(Z)}{\mu_{1/2}(\mathcal{N}_\rho (Y) )} (1 -\eta(Z))
\left( \varphi (x(Z)) m_{1/2}(Z) - \varphi(x(Y))m_{1/2}(Y) \right) \notag \\
&\quad - 
\frac{1}{\omega_d \rho^d }
\sum_{Y \in \mathbb{X}_{k(\rho)}} \eta(Y) V(x(Y))
\sum_{Z \in \mathbb{E}_{k(\rho) Y}} 
\frac{\mu_{1/2}(Z)}{\mu_{1/2}(\mathcal{N}_\rho (Y) )} (1 -\eta(Z)) \notag \\
&\mbox{ } \hspace{25mm} \times
\left( \varphi (x(Z)) m_{1/2}(Z) - \varphi (x(Y))m_{1/2}(Y) \right) \notag \\
&\quad - \frac{1}{\omega_d \rho^d} \sum_{Y \in \mathbb{X}_{k(\rho)}} 
\eta( Y) V(x(Y)) \varphi (x(Y)) m_{1/2}(Y) 
\notag \\
&=:\mathrm{(I)}-\mathrm{(II)}-\mathrm{(III)}.
\end{align}
Noting that
\begin{equation*}
\mathrm{(III)}=\langle \pi_{{\mathbb X}_{k(\rho)}}(\eta), V\varphi \rangle,
\end{equation*}
to complete the proof, we shall show that
\begin{align*}
\mathrm{(I)} - \mathrm{(II)} = \langle \pi_{{\mathbb X}_{k(\rho)}}(\eta), \Delta_1 \varphi \rangle+o(1).
\end{align*}

We start by decomposing (I) into three terms as
\begin{align}
&
\frac{1}{\delta \omega_d \rho^d }
\sum_{Y\in {\mathbb X}_{k(\rho)}}
\eta(Y)
\sum_{Z \in ({\mathbb E}_{k(\rho)})_{Y}}
p_{1/2, 0} (Y,Z)\big(1-\eta(Z) \big)\varphi(x(Z))\big(m_{1/2} (Z)-m_{1/2} (Y)\big)
\nonumber \\
&\quad +
\frac{1}{\delta\omega_d \rho^d }
\sum_{Y\in {\mathbb X}_{k(\rho)}}
\eta(Y)m_{1/2} (Y)
\sum_{Z \in ({\mathbb E}_{k(\rho)})_{Y}}
p_{1/2, 0} (Y,Z)\big(\varphi(x(Z))-\varphi(x(Y))\big)
\nonumber \\
&\quad -
\frac{1}{\delta \omega_d \rho^d }
\sum_{Y\in {\mathbb X}_{k(\rho)}}
\eta(Y)m_{1/2} (Y)
\sum_{Z \in ({\mathbb E}_{k(\rho)})_{Y}}
p_{1/2, 0 } (Y,Z)\eta(Z)\big(\varphi(x(Z))-\varphi(x(Y))\big)
\nonumber \\
&=:{\rm{(Ia)}}+{\rm{(Ib)}}-{\rm{(Ic)}}.
\end{align}
We estimate (Ia), (Ib) and (Ic) separately.

First, let us estimate (Ia).
Taking $\rho >0$ small enough so that 
\begin{equation} 
\frac{ \mu_{1/2}(\mathcal{N}_\rho (X))}{ \omega_d \rho^d} \leq 2w^{1/2}(x(X)) \quad X \in 
\mathbb{X}_{k(\rho)} \mbox{ with }
X \subset  U_{2\rho}\left( \mathrm{supp} \varphi \right)
\label{rho nbd}
\end{equation}
(see (\ref{nbd}) for details), we obtain
\begin{align}
\vert {\rm{(Ia)}} \vert &\leq \frac{1}{\delta \omega_d \rho^{d}}\sum_{Y\in {\mathbb X}_{k(\rho)}}
\sum_{Z \in (\mathbb{E}_{k(\rho)})_{Y}} \vert \varphi(x(Z)) \vert 
p_{1/2, 0}(Y,Z) m_{1/2} (Y) \Big \vert \frac{m_{1/2}(Z)}{m_{1/2} (Y)}-1 \Big \vert
\nonumber \\
&\leq \frac{1}{\delta \omega_d \rho^{d}}\sum_{Y\in {\mathbb X}} 
\max_{ Z \in (\mathbb{E}_{k(\rho)} )_Y }  
\left\vert \varphi(x(Z)) 
\Big( \frac{m_{1/2}(Z)}{m_{1/2} (Y)}-1 \Big) \right\vert 
\Big( \sum_{Z\in (\mathbb{E}_{k(\rho)})_Y}p_{1/2, 0}(Y,Z) \Big) m_{1/2}(Y)
\nonumber \\
&\leq \frac{2}{\delta}\Vert  \varphi \Vert_{\infty} \cdot
\mu_{1}\big( U_{2}\left( \mathrm{supp} \varphi \right) \big) 
\nonumber \\
&\quad \times
\sup
\left\{ 
\Big \vert \frac{m_{1/2}(Z)}{m_{1/2}(Y)}-1 \Big \vert  
\, ; \, Y,Z \in {\mathbb X}, Y\subset U_{2}\left( \supp \varphi \right), Z\in (\mathbb{E}_{k(\rho)} )_Y
\right\}.
\end{align}
By taking $r>0$ satisfying $U_{2} (\supp \varphi)\subset B(o, r)$,
condition {\bf (C)}
concludes that $| (\mathrm{Ia})| =o(1)$ $(\rho \searrow 0)$ via Lemma \ref{hobo 1/2}.

Next, let us compute (Ib).  
Combining $\vert {\mathbb X}_{k(\rho)} \vert=o(\rho^3)$ with
the key estimate (see Lemma \ref{one particle}, and also \cite[Theorem 3.2]{IK24} for the Riemannian case):
\begin{align*}
&
\Big \vert \frac{1}{\delta}
\sum_{Z \in ({\mathbb E}_{k(\rho)})_{Y}}
p_{1/2, 0} (Y,Z)\big(\varphi(x(Z))-\varphi(x(Y))\big)-
\left( \Delta_1 \varphi \right) (x(Y))
\Big \vert \\
&\quad 
\leq K_{1} \frac{\vert {\mathbb X}_{k(\rho)} \vert}{\rho^2}+K_{2} \rho=o(1), 
\quad Y\in {\mathbb X}_{k(\rho)},
\end{align*}
we obtain
\begin{align}
{\rm{(Ib)}}&=
\frac{1}{ \omega_d \rho^d }
\sum_{Y\in {\mathbb X}_{k(\rho)}, Y\subset U_{2\rho}\left( \mathrm{supp} \varphi \right) }
\eta(Y)m_{1/2} (Y) 
\big(\left( \Delta_1 \varphi \right) (x(Y)) +o(1) \big)
\nonumber \\
&=\langle 
 \pi_{{\mathbb X}_{k(\rho)}}(\eta), \Delta_1 \varphi \rangle
+o(1).
\end{align}

Finally, let us show $\mathrm{(Ic)}=0$. Indeed,
by the definition of the proximity graph $\mathbb{G}(\mathbb{X}_{k(\rho)}, \rho)
=(\mathbb{X}_{k(\rho)}, \mathbb{E}_{k(\rho)})$, 
$e=(Y, Z) \in \mathbb{E}_{k(\rho)}$, $Y,  Z \in \mathbb{X}_{k(\rho)}$ 
implies that $\overline{e}=(Z, Y) \in \mathbb{E}_{k(\rho)}$. 
Moreover, recalling (\ref{symmetry}), we have
\begin{align}
{\rm{(Ic)}}&=
\frac{1}{\delta \omega_d \rho^d }
\sum_{Y\in {\mathbb X}_{k(\rho)}}
\sum_{Z \in ({\mathbb E}_{k(\rho)})_{Y}}
\eta(Y)\eta(Z) \mu_{1/2}(Y)\mu_{1/2} (Z)
\big(\varphi (x(Z))-\varphi(x(Y)) \big)
\nonumber \\
&=\frac{1}{\delta \omega_d \rho^d }
\sum_{e\in {\mathbb E}_{k(\rho)}} \eta(t(e))\eta(o(e))
\mu_{1/2}(o(e))\mu_{1/2} (t(e))
\big( \varphi (x(t(e)))-\varphi(x(o(e))) \big)
\nonumber \\
&=
\frac{1}{\delta \omega_d \rho^d }
\sum_{e\in {\mathbb E}_{k(\rho)}} \eta(t({\overline e}))\eta(o({\overline e}))
\mu_{1/2} (o({\overline e})) \mu_{1/2} (t({\overline e}))
\big( \varphi (x(t(\overline{e})))-\varphi(x(o(\overline{e}))) \big)
\nonumber \\
&=\frac{1}{\delta \omega_d \rho^d }
\sum_{e\in {\mathbb E}_{k(\rho)}} \eta(o(e))\eta(t(e))
\mu_{1/2} (t(e))\mu_{1/2} (o(e))
\big( \varphi (x(o(e)))-\varphi(x(t(e))) \big)
\nonumber \\
&
=-{\rm{(Ic)}},
\end{align}
which implies ${\rm{(Ic)}}=0$.

To conclude the proof, it remains to show that $|\mathrm{(II)}|=o(1)$.
Taking $\rho>0$ satisfying (\ref{rho nbd}), by condition {\bf{(C)}} 
together with Lemma \ref{hobo 1/2}, we obtain
\begin{align*}
|\mathrm{(II)}|  &\leq 
\frac{1}{\omega_d \rho^d }
\sum_{Y \in \mathbb{X}_{k(\rho)}}  V(x(Y)) m_{1/2}(Y) \\
&\quad \times 
\sum_{Z \in \mathbb{E}_{k(\rho) Y}} 
\frac{\mu_{1/2}(Z)}{\mu_{1/2}(\mathcal{N}_\rho (Y) )} 
\left(
\left| \varphi (x(Z)) - \varphi (x(Y)) \right| 
+
 |\varphi(x(Z))| \left| \frac{m_{1/2}(Z)}{m_{1/2}(Y)} -1 \right| 
\, \right) \\
& \leq 
2  \| V \|_{\infty, U_{2\rho}\left( \mathrm{supp} \varphi \right)} \cdot \mu_1(U_{2\rho}\left( \mathrm{supp} \varphi \right)) 
\bigg (
\| \nabla \varphi \|_\infty \rho
\nonumber \\
& \quad 
+
\| \varphi \|_{\infty} 
\sup \bigg \{ 
\left| \frac{m_{1/2}(Z)}{m_{1/2}(Y)} -1 \right| ; Y\in \mathbb{X}_{k(\rho)}, 
Y \subset U_{2\rho}\left( \mathrm{supp} \varphi \right), Z \in 
(\mathbb{E}_{k(\rho)} )_Y \bigg \} 
\bigg)
 \\
&\leq  C \rho.
\end{align*}
Combining the above estimates, we conclude the lemma.
\end{proof}

Recall the definition \eqref{carreduchamp} of the carr\'e du champ $\Gamma^{\rho}(\eta ;\varphi)$.
In the next lemma, we compute the carr\'e du champ $\Gamma^{\rho}(\eta ;\varphi)$ 
and show that the time integral of the carr\'e du champ vanishes
as $\rho\searrow0$.

\begin{lm}\label{lm2}
For any $\varphi\in C_c^\infty(M)$, any
 $\rho >0$ such that $\delta V(x(Y)) <1$ for all $Y \in \mathbb{X}_{k(\rho)}$
with $Y \subset U_{2\rho}\left( \mathrm{supp} \varphi \right) $ and 
any $\eta \in  \mathcal{X}_{k(\rho)}$, 
 we have
\begin{align}
\Gamma^{\rho}(\eta ; \varphi)=&: \Gamma_{\mathrm{ex}}^\rho (\eta ; \varphi)
+\Gamma_{\mathrm{kill}}^\rho (\eta ; \varphi)\\
 =&\frac{1}{\delta} \sum_{Y\in {\mathbb X}_{k(\rho)}}
\sum_{Z\sim_\rho Y} \eta(Y)
\big( 1- \eta(Z) \big)
\big( 1- \delta V(x(Y)) \big)
\frac{ \mu_{1/2}(Z)}{ \mu_{1/2}(\mathcal{N}_{\rho}(Y))}
\nonumber \\
& \qquad \times \Big \{ \frac{1}{\omega_{d}\rho^d} \big(
\varphi (x (Z)) m_{1/2} (Z)-\varphi (x (Y))m_{1/2} (Y)\big)
\Big \}^2
\nonumber \\
&  {}+ \sum_{Y\in {\mathbb X}_{k(\rho)}} V(x(Y))\, \eta(Y)
\Big\{ \frac{1}{\omega_{d}\rho^{d}}\, \varphi(x(Y))\, m_{1/2}(Y) \Big\}^{2}. 
\label{Carre}
\end{align}
Consequently,
\begin{align}\label{Carre2}
\lim_{\rho \searrow0} \mathbb{E} \left[ \int_0^T \Gamma^{\rho}(\xi^\rho (s) ; \varphi) ds \right] =0.
\end{align}
\end{lm}
\begin{proof}
We first note that, by a direct computation from \eqref{carreduchamp},
the carr\'e du champ $\Gamma^{\rho}(\eta;\varphi)$ is expressed as
\begin{equation*}
\Gamma^{\rho}(\eta;\varphi)
=\frac{1}{\delta}\sum_{e\in \mathbb{E}^{\partial}_{k(\rho)}}
p_{1/2,\delta}(e)\, \eta(o(e))
\big( f_{\rho}(\eta^{e};\varphi)-f_{\rho}(\eta;\varphi) \big)^{2},
\end{equation*}
where the summation runs not only over the exchange jumps
$\eta \to \eta^{Y,Z}$, $(Y,Z)\in \mathbb{E}_{k(\rho)}$, but also over the killing jumps
$\eta \to \eta^{Y,X(\partial)}$, $Y \in \mathbb{X}_{k(\rho)}$.

By the choice of $\rho>0$, for any $Y, Z \in \mathbb{X}_{k(\rho)}$
with $Y\subset  U_{2\rho}\left( \mathrm{supp} \varphi \right) $ and $Y\sim_{\rho}Z$, we see
\begin{equation*}
p_{1/2, \delta}(Y, Z)= (1-\delta V(x(Y)) )
\frac{\mu_{1/2} (Z)}{{\mu}_{1/2} ({\mathcal N}_{\rho}(Y))} ,
\end{equation*}
and hence the exchange jumps give rise to the first term on the right-hand side of
\eqref{Carre}. On the other hand, the killing jump $\eta \to \eta^{Y,X(\partial)}$
occurs with rate $p_{1/2,\delta}(Y,X(\partial))=\min\{\delta V(x(Y)),1\}$
and changes $f_{\rho}$ by
\begin{equation*}
f_{\rho}(\eta^{Y,X(\partial)};\varphi)-f_{\rho}(\eta;\varphi)
=-\frac{1}{\omega_{d}\rho^{d}}\, \varphi(x(Y))\, m_{1/2}(Y).
\end{equation*}
This increment vanishes unless $x(Y)\in \mathrm{supp}\, \varphi$, in which case
$Y \subset U_{2\rho}\left( \mathrm{supp} \varphi \right)$ and thus
$\min\{\delta V(x(Y)),1\}=\delta V(x(Y))$ by the choice of $\rho$.
Therefore, the killing jumps give rise to the second term on the right-hand side of
\eqref{Carre}. Here the prefactor $\delta^{-1}$ is cancelled by the killing rate
$\delta V(x(Y))$. This proves \eqref{Carre}.

Next, let us prove \eqref{Carre2}.
We begin with the exchange part $\Gamma^{\rho}_{\mathrm{ex}}$.
Note that the summands of $\Gamma^{\rho}_{\mathrm{ex}}$ vanish unless
$Y \subset U_{2\rho}\left( \mathrm{supp} \varphi \right)$.
Since $\xi^{\rho}(s,Y)\in \{0,1\}$ and $0 \leq 1-\delta V(x(Y) )\leq 1$, we obtain
\begin{align*}
|\Gamma^{\rho}_{\mathrm{ex}}(\xi^\rho (s) ; \varphi)|
& \leq
\frac{1}{\delta} \sum_{\substack{Y\in {\mathbb X}_{k(\rho)} , \\
Y\subset U_{2\rho}\left( \mathrm{supp} \varphi \right)} }
\sum_{Z\sim_\rho Y}
\frac{ \mu_{1/2}(Z)}{ \mu_{1/2}(\mathcal{N}_{\rho}(Y))}
\left( \frac{m_{1/2} (Y)}{\omega_d \rho^d} \right)^2
\nonumber \\
& \quad \times 2\Big \{ \varphi (x (Z))^2
 \left( \frac{ m_{1/2} (Z)}{m_{1/2}(Y)}-1 \right)^2 +|\varphi(x(Z))-\varphi (x (Y))|^2
 \Big \}.
\end{align*}
Here we note that
\begin{equation*}
|\varphi(x(Z))-\varphi (x (Y))|\leq 2\| \nabla \varphi \|_{L^\infty (U_{2\rho}\left( \mathrm{supp} \varphi \right) )} \rho
\end{equation*}
for $Y \subset U_{2\rho}\left( \mathrm{supp} \varphi \right)$ and $Z \sim_{\rho} Y$.
By Lemma \ref{nbd}, for sufficiently small $\rho>0$ and for any $Y \subset U_{2\rho}\left( \mathrm{supp} \varphi \right) $,
\begin{equation*}
\frac{m_{1/2} (Y)}{\omega_d \rho^d} \leq 2 \mu_1 (Y) \leq 2 \mu_1(B(x(Y), |\mathbb{X}_{k(\rho)}| ))
\leq 4 w(x(Y))\omega_d |\mathbb{X}_{k(\rho)} |^d.
\end{equation*}
Moreover, under condition {\bf (C)}, Lemma \ref{hobo 1/2} implies that
\begin{align*}
&|\Gamma^{\rho}_{\mathrm{ex}}(\xi^\rho (s) ; \varphi)|\\
&\quad \leq
o_{B(o,r+2)}(\rho^{3d-2}) 
\\
&\quad \quad
\times
\sum_{\substack{Y\in {\mathbb X}_{k(\rho)} , \\
Y\subset U_{2\rho}\left( \mathrm{supp} \varphi \right)} }
\max_{Z \sim_{\rho} Y} \Big\{ \varphi (x (Z))^2
 \left( \frac{ m_{1/2} (Z)}{m_{1/2}(Y)}-1 \right)^2 +
 \| \nabla \varphi \|_{L^\infty (U_{2\rho}\left( \supp \varphi \right) )}^2 \rho^2 \Big\}
 \mu_1(Y)\\
&\quad \leq o_{B(o,r+2)}(\rho^{3d}),
\end{align*}
where 
$r>0$ is a radius so that $U_2(\supp \varphi) \subset B(o,r)$ and
we also used
\begin{align*}
\sum_{Y \subset U_{2\rho}\left( \mathrm{supp} \varphi \right)} \mu_{1}(Y)
\leq \mu_{1}\big(U_{2\rho}(\mathrm{supp}\, \varphi)\big)<\infty.
\end{align*}

It remains to estimate the killing part $\Gamma^{\rho}_{\mathrm{kill}}$.
Since the summands of $\Gamma^{\rho}_{\mathrm{kill}}$ vanish unless
$x(Y)\in \mathrm{supp}\, \varphi$, using the estimate of
$m_{1/2}(Y)/(\omega_{d}\rho^{d})$ displayed above and
$$\sum_{Y:\, x(Y)\in \mathrm{supp} \varphi}\, m_{1/2}(Y)/(\omega_{d}\rho^{d})
\leq 2\mu_{1}\big(U_{2\rho}\left( \mathrm{supp} \varphi \right)\big),$$
as in the proof of Lemma \ref{lm1},
we obtain
\begin{align*}
0 &\leq \Gamma^{\rho}_{\mathrm{kill}}(\xi^{\rho}(s);\varphi)
\\
&\leq \Vert V \Vert_{\infty, U_{2\rho}\left( \mathrm{supp} \varphi \right)} \Vert \varphi \Vert_{\infty}^{2}
\Big( \max_{Y:\, x(Y)\in \mathrm{supp} \varphi} \frac{m_{1/2}(Y)}{\omega_{d}\rho^{d}} \Big)
\sum_{Y:\, x(Y)\in \mathrm{supp} \varphi} \frac{m_{1/2}(Y)}{\omega_{d}\rho^{d}}
\\
&\leq 8 \Vert V \Vert_{\infty, U_{2\rho}\left( \mathrm{supp} \varphi \right)} \Vert \varphi \Vert_{\infty}^{2}
\Vert w \Vert_{\infty, U_{2\rho}\left( \mathrm{supp} \varphi \right)}\, \omega_{d}\,
\vert \mathbb{X}_{k(\rho)} \vert^{d}\,
\mu_{1}\big(U_{2\rho}\left( \mathrm{supp} \varphi \right)\big)\\
&= o_{B(o,r+2)}(\rho^{3d}),
\end{align*}
where $r>0$ is as above. 
Combining the estimates of $\Gamma^{\rho}_{\mathrm{ex}}$ and
$\Gamma^{\rho}_{\mathrm{kill}}$, we conclude \eqref{Carre2}.
\end{proof}

\subsection{Tightness and characterization of limit points}

Recall the probability measure ${\mathbb Q}^\rho$ defined at the beginning of this section.
In this subsection, we prove the relative compactness of $\{{\mathbb Q}^\rho\}_{\rho>0}$
and characterize its limit points. We first establish relative compactness.

\begin{lm}\label{tight}
The family of probability measures $\{{\mathbb Q}^\rho\}_{\rho>0}$ is relatively compact.
\end{lm}
\begin{proof}
It follows from \cite[Theorem 3.1]{Jak86} that to show the relative compactness of $\{{\mathbb Q}^\rho\}_{\rho>0}$
it is enough to show the following two conditions:
\begin{itemize}
\item[{\rm (i)}] There exists a compact set $\mc C\subset \mc M_+(M)$
such that for any sufficiently small $\rho>0$
$$
{\mathbb Q}^{\rho}(\pi_t\in\mc C \text{ for any $0\le t \le T$})=1.
$$
\item[{\rm (ii)}] For each $\varphi\in C_c^\infty(M)$, the family of random variables $\{\big \langle
\pi_{{\mathbb X}_{k(\rho)}}
\big( \xi^\rho(\cdot)\big), \varphi
\big \rangle\}_{\rho>0}$ is tight in $D([0,T], \bb R)$.
\end{itemize}

To prove (i), let $K_n$ be the closure of $B(o,n)$ for each $n\in \bb N$.
It is clear that $K_n$ is compact, $M=\bigcup_{n\in\bb N} K_n$ and 
$\mu_1(K_n)<\infty$ for all $n\in \mathbb N$.
It follows from \eqref{nbd} that
\begin{align*}
c_n:=\sup_{0<\rho\le1}\sup_{X\in {\mathbb X_{k(\rho)}}: x(X)\in K_n}\frac{\mu_{1/2} (\mathcal{N}_\rho (X))}{\omega_d \rho^d} < \infty,
\end{align*} 
for any $n\in\bb N$.
Therefore, for any sufficiently small $\rho>0$ and
any configuration $\eta\in\mc X_{k(\rho)}$, we have
\begin{align*}
\lan\pi_{\bb X_{k(\rho)}}(\eta), {\bf 1}_{K_n}\ran &\le 
\frac{1}{\omega_{d} \rho^{d}} \sum_{X\in {\mathbb X}: x(X)\in K_n}
\mu_{1/2}(X)\mu_{1/2}(\mc N_{\rho}(X))\\
&\le c_n\mu_{1/2}(K_{n+1})=:d_n<\infty.
\end{align*}
Let $\mc K$ be the set of all Radon measures $\pi\in\mc M_+(M)$ satisfying $\pi(K_n)\le d_n$ for each $n\in\bb N$.
Since $\sup_{\pi\in \mc K} \pi(C) < \infty$
for any compact set $C\subset M$, $\mc K$ is relatively compact (cf. \cite[Theorem 4.2]{Kal17}).
Letting $\mc C$ be the closure of $\mc K$, we obtain (i).

To prove (ii), we use the Dynkin martingale \eqref{dynkin}.
Fix a smooth function $\varphi\in C_c^\infty(M)$. It follows from Lemma \ref{lm1} that
$\big \langle
\pi_{{\mathbb X}_{k(\rho)}}
\big( \xi^\rho(t)\big), \varphi
\big \rangle$
can be rewritten as
\begin{align*}
M^\rho(\varphi)_{t}
+
\big \langle
\pi_{{\mathbb X}_{k(\rho)}}
\big( \xi^\rho(0)\big), \varphi
\big \rangle
+\int_{0}^{t} 
\langle
\pi_{{\mathbb X}_{k(\rho)}}
\big(\xi^\rho(s)\big),
\Delta_{1} \varphi - V\varphi
\big \rangle ds +o(1),
\end{align*}
for any $0\le t \le T$. Therefore, to prove (ii), it is enough to show that
the first three terms on the right-hand side are tight.
The tightness of the martingale term follows from Lemma \ref{lm2}.
The tightness of the second term immediately follows from the assumption on the initial distribution.
The tightness of the integral term follows from the Aldous criterion, see \cite[Proposition 4.1.6]{KL99}.
\end{proof}

We now set
\begin{align*}
{\mf S}:=\Big \{ \pi=&(\pi_{t})_{0\leq t \leq T}\in C([0,T],\mc M_+(M))\, ; \, \pi_{0}(dx)=u_{0}(x)\mu_{1}(dx)~\mbox{and} 
\\
&
\langle \pi_t,\varphi \rangle - \langle u_0,\varphi \rangle
-\int_0^t \langle \pi_s, \Delta_{1} \varphi - V\varphi\rangle ds =0,
\, \forall \varphi\in C_c^\infty(M),\, \forall t\in[0,T] \Big \},
\end{align*}
and let $\mf K_{\rm ac}\subset D([0,T],\mc M_+(M))$ be the set
of all $\pi=(\pi_t)_{0\le t \le T}\in D([0,T], \mc M_+(M))$ such that,
for any $0\le t \le T$,  $\pi_t$ is absolutely continuous with respect to $\mu_1$ and
its Radon--Nikodym density is bounded by $1$ almost everywhere.

The next lemma shows that every limit point of $\{{\mathbb Q}^\rho\}_{\rho>0}$ is concentrated 
on $\mf S\cap \mf K_{\rm ac}$.

\begin{lm}\label{cha}
Let ${\mathbb Q}^*$ be any limit point of the family $\{{\mathbb Q}^\rho\}_{\rho>0}$.
Then we have 
$${\mathbb Q}^*(\mf S\cap \mf K_{\rm ac})=1.$$
\end{lm}
\begin{proof}
We first show ${\mathbb Q}^*(\mf S)=1$.
It follows from Doob's inequality and Lemma \ref{lm2} that
\begin{align*}
\lim_{\rho \searrow 0} {\mathbb E} \left[ \sup_{0\leq t \leq T}
\left \vert M^\rho(\varphi)_{t}
\right \vert^2 \right]=0.
\end{align*}
This limit shows that the martingale term of the decomposition in Lemma \ref{lm1}
vanishes in the limit $\rho\searrow0$. A standard argument then shows that
\begin{align*}
{\mathbb Q}^*\left(  \pi \,\middle|\, \sup_{0\le t \le T} \left| \lan\pi_t,\varphi\ran - \lan u_0,\varphi\ran
-\int_0^t \lan\pi_s, \Delta_{1} \varphi - V\varphi\ran ds \right| > \varepsilon \right)=0,
\end{align*}
for any $\varepsilon>0$ and any $\varphi\in C_c^\infty(M)$.
Choosing positive rational values of $\varepsilon$ and a countable dense subset of
$C_c^\infty(M)$, and then using continuity in $\varphi$, we obtain
\begin{align*}
{\mathbb Q}^*\left(\pi \,\middle|\,  \lan\pi_t,\varphi\ran - \lan u_0,\varphi\ran
-\int_0^t \lan\pi_s, \Delta_{1} \varphi - V\varphi\ran ds =0,
\ \forall\varphi\in C_c^\infty(M),\ \forall t\in[0,T]\right)
= 1.
\end{align*}
The preceding integral identity implies that $t\mapsto \pi_t$ is continuous on $[0,T]$.
Hence ${\mathbb Q}^*(\mf S)=1$. See \cite[Chapter 4]{KL99} for details.

We next show ${\mathbb Q}^*(\mf K_{\rm ac})=1$.
Fix $\varphi\in C_c^\infty(M)$ and note that
\begin{align*}
\sup_{0\le t \le T}\left|\big \langle
\pi_{{\mathbb X}_{k(\rho)}}
\big( \xi^\rho(t)\big), \varphi
\big \rangle \right|\le 
\frac{1}{\omega_{d} \rho^{d}} \sum_{X\in {\mathbb X}: x(X)\in K_n}
\mu_{1/2}(X)\mu_{1/2}(\mc N_{\rho}(X))|\varphi(x(X))|,
\end{align*}
if $\supp\varphi\subset K_n$ for some $n\in\mathbb N$.
It follows from \eqref{nbd} that the right-hand side
converges as $\rho\searrow0$ to
\begin{align*}
\int_M w^{1/2}(x)|\varphi(x)| \mu_{1/2}(dx)
= \int_M |\varphi(x)| \mu_{1}(dx).
\end{align*}
Since ${\mathbb Q}^*$ is a limit point of $\{{\mathbb Q}^{\rho}\}_{\rho>0}$,
${\mathbb Q}^*$ is concentrated on all trajectories $\pi=(\pi_t)_{0\le t \le T}$ satisfying
\begin{align*}
\sup_{0\le t \le T} |\lan\pi_t, \varphi \ran| \le \int_M |\varphi(x)| \mu_1(dx).
\end{align*}
Therefore, we obtain ${\mathbb Q}^*(\mf K_{\rm ac})=1$.
\end{proof}

\subsection{Proof of Theorem \ref{mainthm}}\label{pf of main result}

We prove Theorem \ref{mainthm} in this subsection.
Theorem \ref{mainthm} is an immediate consequence of Lemmas \ref{tight} and \ref{cha} together with the uniqueness result for \eqref{hdleq}.

\begin{proof}[{\bfseries Proof of Theorem \ref{mainthm}}]
By Lemma \ref{tight}, the family $\{{\mathbb Q}^\rho\}_{\rho>0}$ has an accumulation point, 
denoted by ${\mathbb Q}^*$.
By Lemma \ref{cha} and Theorem \ref{thm:uniq},
${\mathbb Q}^*$ is concentrated on the unique bounded weak solution to \eqref{hdleq}.
Therefore the family $\{{\mathbb Q}^\rho\}_{\rho>0}$ converges as $\rho\searrow0$
to the Dirac measure concentrated on the trajectory, which is absolutely continuous and whose density is the unique weak solution to \eqref{hdleq}.
Since the limit point is concentrated on continuous trajectories, Theorem \ref{mainthm} then follows from the continuous mapping theorem.
\end{proof}
\begin{proof}[{\bfseries Proof of Corollary \ref{Co1}}]
Using (\ref{nbd}), we obtain
\begin{equation*}
\langle \widetilde{\pi}_{\mathbb{X}} (\eta), \varphi \rangle=
\left\langle \pi_{\mathbb{X}} (\eta), \frac{\varphi}{w^{1/2}} \right\rangle
+O_{B(o,r+2)}\left( \frac{|\mathbb{X}|}{\rho^2} +\rho^2 \right),
\end{equation*}
where $r>0$ is a radius so that $U_2(\supp \varphi) \subset B(o,r)$.
Since $|\mathbb{X}_{k(\rho)}|=o(\rho^3)$, the initial condition in (\ref{Co1 ini}) implies that 
\begin{align*}
\lim_{\rho \searrow 0} \nu^{k(\rho)}\bigg( \Big \vert 
\big \langle
\pi_{{\mathbb X}_{k(\rho)}}({\xi}^{\rho}(0)),
\frac{\varphi}{w^{1/2}}
\big \rangle
-\int_{M} u_{0}(x) \frac{\varphi (x)}{w^{1/2}(x)} {\mu}_1 (dx) \Big \vert >\varepsilon \bigg)=0
\end{align*}
for any $\varepsilon>0$ and $\varphi \in C^{\infty}_{c}(M)$. 
Theorem \ref{mainthm} yields, for every $0\le t \le T$,
\begin{align*}
\lim_{\rho \searrow 0} {\mathbb P}_{\nu^{k(\rho)}}\bigg( 
\Big \vert 
\big \langle
\pi_{{\mathbb X}_{k(\rho)}}({\xi}^{\rho}(t)),
\frac{\varphi}{w^{1/2}}
\big \rangle
-\int_{M} u(t,x) \frac{\varphi(x) }{w^{1/2}(x)} {\mu}_1(dx) \Big \vert 
>\varepsilon \bigg)=0
\end{align*}
holds for any $\varepsilon>0$ and $\varphi\in C^{\infty}_{c}(M)$.
Using (\ref{nbd}) again, we conclude (\ref{Co1 claim}).
\end{proof}
\begin{proof}[{\bfseries Proof of Corollary \ref{Co2}}]
The assumption (\ref{hatkatei}) implies that
\begin{equation}
\mu_{1/2}(X_{0(k(\rho))}) \langle \hat{\pi}_{\mathbb{X}_{k(\rho)}}(\eta), \varphi \rangle=
\left\langle \widetilde{\pi}_{\mathbb{X}_{k(\rho)}} (\eta) ,\varphi \right\rangle 
+o_{B(o,r+2)}(1)
\quad \mbox{ as }\rho \searrow 0,
\label{hatapprox}
\end{equation}
where $r>0$ is a radius so that $U_2(\supp \varphi) \subset B(o,r)$.
Then the initial condition in (\ref{Co2 ini}) implies (\ref{Co1 ini}), which in turn
implies (\ref{Co1 claim}). Using (\ref{hatapprox}) again, we conclude (\ref{Co2 claim}).
\end{proof}
\section{Uniqueness of bounded weak solutions to \eqref{hdleq}}\label{uni}
In this section, we prove the uniqueness of bounded weak solutions to \eqref{hdleq}
and give the explicit representation of the unique solution
in terms of the minimal Schr{\"o}dinger kernel.
The proof is divided into several steps.
We first extend \eqref{weak heat} to the case where the test function $\varphi$ is time-dependent.
We then prove that every bounded weak solution is smooth on $(0,T)\times M$, a property
known as parabolic regularity. This smoothness allows us to apply It\^o's formula and
obtain a Feynman--Kac type representation. The representation yields the desired
uniqueness and the explicit expression \eqref{explicit-representation}.

Recall that $(M, g, \mu_{\alpha})$ is a weighted Riemannian manifold with weighted measure 
$\mu_{\alpha}(dx)=w^{\alpha}(x)\,\mathrm{vol}_{g}(dx)$. Throughout this section, we restrict to the case $\alpha=1$ 
and, for brevity, write $\mu$, $\Delta_{\mu}$ and $p(t,x,y)$ in place of $\mu_1$, $\Delta_1$ and $p_1(t,x,y)$, 
respectively.
Throughout this section, we identify $u$ with the jointly measurable
version $\widetilde{u}$ given in Remark~2.4. This causes no ambiguity because $u$ appears below only through integrals with respect to
$dt \otimes \mu(dx)$.

As mentioned above, we extend \eqref{weak heat} to the case
where the test function $\varphi$ is time-dependent.

\begin{lm}
\label{lem:equivweak}
Let $\pi_{t}(dx)=u(t,x)\mu(dx)$ be a measure-valued bounded weak solution in the sense of
Definition~\ref{def:weak}.  Then
the map
$t\mapsto \langle \pi_{t}, \varphi \rangle$  
is continuous on $[0,T]$ for every $\varphi\in C_c^\infty(M)$
and 
\begin{align}\label{eq:weak-space-time}
-\int_0^T dt \int_M &
u(t,x)\,\partial_t \zeta (t,x) \mu(dx)
\nonumber \\
&=
\int_0^T dt \int_M u(t,x)\,\big(\Delta_\mu \zeta (t,x) -V(x) \zeta(t,x)\big)\mu(dx)
\end{align}
holds for every $\zeta\in C_c^\infty((0,T)\times M)$.
\end{lm}

\begin{proof}
Since $C_c^\infty(M)\subset C_c(M)$, the continuity of 
$t\mapsto \langle \pi_{t}, \varphi \rangle$ on $[0,T]$
is immediate from $\pi=(\pi_{t})_{0\leq t \leq T}\in C([0,T], \Mplus)$.
For $\varphi\in C_c^\infty(M)$, set
\[
A_\varphi(t):=\langle \pi_{t}, \varphi \rangle,
\qquad
B_\varphi(t):=\langle \pi_{t}, (\Delta_{\mu}-V)\varphi \rangle, \qquad 0\leq t\leq T.
\]
As noted above, $A_\varphi$ is continuous on $[0,T]$, 
and \eqref{RN-unif-bdd} ensures that $B_\varphi$ is bounded on $[0,T]$.
Then \eqref{weak heat} yields
\begin{align}
A_\varphi(t)=A_\varphi(0)+\int_0^t B_\varphi(s)\,ds, \quad 0\leq t \leq T,
\label{May24-1}
\end{align}
so $A_\varphi$ is absolutely continuous 
and $A_\varphi'(t)=B_\varphi(t)$ for a.e.\ $t\in[0,T]$.

Now fix $\eta\in C_c^\infty((0,T))$ and set $\zeta(t,x):=\eta(t)\varphi(x)$. 
Since $\eta(0)=\eta(T)=0$,
we multiply \eqref{May24-1} by $\eta'(t)$ and integrate over $[0,T]$ to obtain
\[
\int_0^T A_\varphi(t)\eta'(t)\,dt
= -\int_0^T B_\varphi(t)\eta(t)\,dt,
\]
which is exactly \eqref{eq:weak-space-time} for $\zeta=\eta \varphi$.
By linearity, \eqref{eq:weak-space-time}
extends to every element of the algebraic tensor
product $C_c^\infty((0,T))\otimes C_c^\infty(M)$.

We recall that the
algebraic tensor product \(C_c^\infty((0,T))\otimes C_c^\infty(M)\)
is dense in \(C_c^\infty((0,T)\times M)\) with respect to the standard
test-function topology. Indeed, after localizing the spatial variable by
a finite partition of unity on the compact spatial projection of the
support, this elementary density fact reduces to the Euclidean case 
(see e.g., \cite[Theorem 39.2]{Tre67}).
Then for any $\zeta \in C^{\infty}_{c}((0,T)\times M)$, 
there exists a sequence $\{ \zeta_{n} \}_{n=1}^{\infty} \subset C_c^\infty((0,T))\otimes C_c^\infty(M)$ 
such that $\partial_{t}\zeta_{n}$ and $\Delta_{\mu}\zeta_{n}$ 
converge to $\partial_{t}\zeta$ and $\Delta_{\mu}\zeta$, respectively,
uniformly on each compact subset $K\Subset (0,T)\times M$.
Combining this with the uniform bound $u\in L^{\infty}((0,T)\times M)$, 
we may pass to the limit in \eqref{eq:weak-space-time} for $\zeta_{n}$ 
to obtain the identity for $\zeta\in C^{\infty}_{c}((0,T)\times M)$.
\end{proof}
\subsection{Local parabolic regularity}
In this subsection, we show that bounded distributional solutions of our
parabolic equation (\ref{eq:weak-space-time}) are smooth in the interior. The proof of the following
lemma relies on the interior regularity theory for \emph{distributional}
solutions of parabolic equations developed in \cite[Section 6.4]{Gri09},
which fits our setting particularly well since the operator $\Delta_\mu$ is
symmetric with respect to $\mu$. 
\begin{lm}
\label{lem:interior-smooth}
If $v\in L^\infty((0,T)\times M)$ satisfies (\ref{eq:weak-space-time})
for every $\zeta\in C_c^\infty((0,T)\times M)$, then
$
v \in C^\infty((0,T)\times M)$.
\end{lm}
\begin{proof}
Fix $(t_0,x_0)\in(0,T)\times M$ and choose $0<a<t_0<b<T$ and a coordinate
neighborhood $U\subset M$ of $x_0$ with relatively compact closure. We set
$Q:=(a,b)\times U$, identified with an open subset of $\R\times\R^d$. 
The weak form (\ref{eq:weak-space-time}) restricted to test functions
$\zeta\in C_c^\infty(Q)\subset C_c^\infty((0,T)\times M)$ becomes
\begin{equation}
(\partial_t -(\Delta_\mu-V))v=0 \quad \mbox{in }   \mathcal{D}^\prime (Q).
\label{parab dist}
\end{equation}

Let $(x^1,\dots,x^d)$ be the local
coordinates on $U$, and write $(g_{ij})$ for the components of the metric
$g$ in these coordinates and $(g^{ij}):=(g_{ij})^{-1}$ on $U$. Since
$\mu(dx)=w(x)\,\mathrm{vol}_g(dx)$ with the smooth positive weight $w$, and
since $\mathrm{vol}_g(dx)=\sqrt{\det g(x)}\,dx$ in the chart, the measure
$\mu$ has the smooth positive density
\[
\rho(x):=w(x)\sqrt{\det g(x)},
\qquad
\mu(dx)=\rho(x)\,dx,
\]
with respect to the Lebesgue measure $dx$ of the chart. Accordingly, the
weighted Laplacian takes the divergence form
\[
\Delta_\mu\varphi
=\frac{1}{w\sqrt{\det g}}\sum_{i,j=1}^{d}
\partial_i\Bigl(w\sqrt{\det g}\,g^{ij}\,\partial_j\varphi\Bigr)
=\frac{1}{\rho}\sum_{i,j=1}^{d}\partial_i\bigl(\rho g^{ij} \partial_j\varphi\bigr),
\qquad \varphi\in C^\infty(U).
\]
We now set
\[
 a^{ij}:=\rho g^{ij}=w\sqrt{\det g} g^{ij}, \qquad i,j=1,\ldots, d,
\]
and introduce the parabolic differential operator
\[
\mathcal{P}:=\rho\,\partial_t
-\sum_{i,j=1}^{d}\partial_i\big( a^{ij}\partial_j \bigr)
\]
in the sense of \cite[Section 6.4.2]{Gri09}, where
the coefficients $\rho>0$ and
$a^{ij}=a^{ji}$ are smooth on $U$, independent of $t$, and
$(a^{ij})$ is positive definite, as required there. 
Then the equation (\ref{parab dist})
can be rewritten as
\begin{align}\label{eq:Gri-eq}
\mathcal{P}v=- \rho V v\qquad\text{in }\mathcal{D}'(Q).
\end{align}
Since $\rho$ and $V$ are smooth, the assumption $v \in L^\infty ((0,T)\times U)$ implies that  
$-\rho V v \in V^{0}_{loc}(Q)$ (see \cite[6.4.1]{Gri09} for the definition of the anisotropic Sobolev space $V^k_{loc}(Q)$). Applying \cite[Theorem 6.25]{Gri09}, we obtain 
$v \in V^{2}_{loc}(Q)$. By the usual bootstrap argument, 
we obtain $v \in V^{\infty}_{loc}(Q)$. The Sobolev embedding theorem
(see \cite[Theorem 6.1]{Gri09}) then yields $v \in C^\infty(Q)$.
\end{proof}
\subsection{A Feynman--Kac type formula}
Let $(X_t)_{t\geq 0}$ be the (weighted) Brownian motion
on the weighted manifold $(M, g,\mu)$ generated by the
weighted Laplacian $\Delta_\mu$. We consider a function 
$z\in C^\infty((0,T)\times M)\cap L^\infty((0,T)\times M)$ satisfying
\[
\partial_t z(t,x) = \Delta_\mu z(t,x) - V(x)z(t,x), \quad (t,x)\in (0,T)\times M.
\]
Using the parabolic regularity of bounded weak solutions discussed in Section 4.1
together with arguments from stochastic analysis,
we obtain the following {\it{Feynman--Kac type formula}}, which plays a key role in the proof of 
Theorem \ref{thm:uniq}. 
\begin{lm}\label{lem:ito-detailed}
Assume that $(M, g, \mu)$ is stochastically complete. Then for 
fixed $t\in(0,T)$, $\varepsilon\in(0,t)$ and $x\in M$, 
we have
\begin{equation}\label{eq:stopped-FK}
z(t,x)
=
\E_x\!\left[
\exp\!\Bigl(-\int_0^{t-\varepsilon}V(X_r)\,dr\Bigr)
 z\bigl(\varepsilon, X_{t-\varepsilon}\bigr)
\right],
\end{equation}
where ${\mathbb E}_{x}$ denotes the expectation with respect to the probability law ${\mathbb P}_{x}$.
\end{lm}
\begin{proof}
First of all, we put $S:=t-\varepsilon$.  
For the fixed initial point $x\in M$, let
$
\mathscr F_s^X
:=
\sigma(X_r;\,0\leq r\leq s),
s\geq0,
$
be the natural filtration generated by $X$.
All martingales below are understood with respect to
$(\mathscr F_s^X)_{s\geq0}$ under $\mathbb P_x$.

We prepare, in the precise form needed below,
the stopped time-dependent Dynkin formula for the 
(weighted) Brownian motion.
Let $D\Subset M$ be an open set, 
\[
\tau_D:=\inf\{s\geq0:X_s\notin D\},
\qquad
\theta_s^D:=s\wedge\tau_D,
\quad 0\leq s\leq S,
\]
$I\subset\mathbb R$ be an open interval containing $[0,S]$ and $F\in C^{\infty}(I\times M)$.
Suppose that $x\in D$ and that there exists a compact set
$K\Subset M$ such that $\operatorname{supp}F(r,\cdot)\subset K$ for every
$r\in I$.  Then
\begin{equation}\label{eq:intrinsic-time-dynkin}
\mathcal N_s^{F,D}
:=F\bigl(\theta_s^D,X_{\theta_s^D}\bigr)-F(0,x)\\
-\int_0^{\theta_s^D}
  (\partial_r+\Delta_\mu)F(r,X_r)\,dr,
\qquad 0\leq s\leq S
\end{equation}
is a continuous martingale under $\mathbb P_x$.
For completeness, we give a proof of this formula without introducing coordinates or
a global frame.  Since $X$ is the diffusion generated by $\Delta_\mu$, for
every $f\in C_c^\infty(M)$,
\[
M_s^{f,D}
:=f\bigl(X_{\theta_s^D}\bigr)-f(x)
  -\int_0^{\theta_s^D}\Delta_\mu f(X_r)\,dr,
\qquad 0\leq s\leq S,
\]
is a continuous local 
martingale under $\mathbb P_x$.
Moreover, noting
\[
\sup_{0\leq s\leq S}|M_s^{f,D}|
\leq 2\|f\|_\infty+S\|\Delta_\mu f\|_\infty,
\]
it is in fact a bounded continuous martingale.  

For $F(r,y)=\phi(r)f(y)$ with $\phi \in C^\infty(I)$ and
$f\in C_c^\infty(M)$, we put $Y_s:=f(X_{\theta_s^D})$ and
$Z_s:=\phi(\theta_s^D)$.
By the definition of $M^{f,D}$, we see
\[
dY_s
=
{\bf 1}_{\{s<\tau_D\}}\Delta_\mu f(X_s)\,ds
+dM_s^{f,D}, \quad
dZ_s
=
{\bf 1}_{\{s<\tau_D\}}\psi'(s)\,ds.
\]
Since $Z$ is a continuous process of finite variation, It\^{o}'s formula gives
\begin{align*}
F(\theta_{s}^{D}, X_{\theta_s^D})-F(0, x)
={}&
\phi(\theta_s^D)f(X_{\theta_s^D})-\phi(0)f(x)
\\
={}&
\int_0^{\theta_s^D}
\phi'(r)f(X_r)\,dr 
+
\int_0^{\theta_s^D}
\phi(r)\Delta_\mu f(X_r)\,dr
+
\int_0^s \phi(r)\,dM_r^{f,D},
\end{align*}
where we used the fact that $M^{f,D}$ is constant after $\tau_D$.
Combining this with
\[
(\partial_r+\Delta_\mu)F(r,y)
=
\phi'(r)f(y)+\phi(r)\Delta_\mu f(y),
\]
we obtain
\[
\mathcal N_s^{F,D}
=
\int_0^s \phi(r)\,dM_r^{f,D},
\qquad 0\leq s\leq S,
\]
and hence \eqref{eq:intrinsic-time-dynkin} is a continuous martingale for every
finite sum of functions of the form $\phi(r)f(y)$.
For a general $F$ as above, choose a function
$\psi \in C_c^\infty(I)$ such that $\psi \equiv 1$
on an open neighborhood of $[0,S]$.  
The sole purpose of the cutoff $\psi$ is to make
$\psi F\in C_c^\infty(I\times M)$.
Since $\psi \equiv 1$ on a neighborhood of $[0,S]$,
we have
\[
\mathcal N_s^{\psi F,D}
=
\mathcal N_s^{F,D},
\qquad 0\leq s\leq S.
\]
The standard
tensor-product density for test functions yields a sequence
\[
F_n(r,y)=\sum_{j=1}^{m_n}\phi_{n,j}(r)f_{n,j}(y),
\qquad
\phi_{n,j}\in C_c^\infty(I),
\quad
f_{n,j}\in C_c^\infty(M),
\]
which converges to $\psi F$ in the test-function topology of
$C_c^\infty(I\times M)$.  In particular, on
$[0,S]\times\overline D$, we have
\[
F_n\longrightarrow F,
\qquad
\partial_rF_n\longrightarrow\partial_rF,
\qquad
\Delta_\mu F_n\longrightarrow\Delta_\mu F \quad \mbox{uniformly as } n\to \infty.
\]
This is the same smooth tensor-product approximation argument used in the
proof of Lemma~\ref{lem:equivweak}.  Consequently,
\begin{align*}
\sup_{0\leq s\leq S}
\bigl|\mathcal N_s^{F_n,D}-\mathcal N_s^{F,D}\bigr|
&\leq
2\|F_n-F\|_{L^\infty([0,S]\times\overline D)}
+S\|(
\partial_r+\Delta_\mu)(F_n-F)
\|_{L^\infty([0,S]\times\overline D)}
\\
&
\longrightarrow0 \quad \mbox{as } n\to \infty.
\end{align*}
The convergence is deterministic and uniform in $s$.  
Hence, for $0\leq u\leq s\leq S$ and every bounded
$\mathscr F_u^X$-measurable random variable $G$,
we may pass to the limit in
\[
\mathbb E_x\bigl[G(
\mathcal N_s^{F_n,D}-\mathcal N_u^{F_n,D})\bigr]=0
\]
to obtain the corresponding identity for $\mathcal N^{F,D}$.  This proves
that $\mathcal N^{F,D}$ is a continuous martingale under ${\mathbb P}_{x}$.

We now choose a nested smooth exhaustion
\begin{equation}\label{eq:nested-exhaustion}
D_1\Subset D_2\Subset\cdots\Subset M,
\qquad
\bigcup_{R=1}^\infty D_R=M,
\end{equation}
where every $D_R$ has smooth boundary.  If $M$ is compact, we may simply
take $D_R=M$ for every $R$.  
For each $R\geq1$, we set
\[
\tau_R:=\inf\{s\geq0:X_s\notin D_R\}.
\]
Since \eqref{eq:nested-exhaustion} exhausts $M$, $(\tau_{R})_{R\geq 1}$ is non-decreasing.

We next choose $\chi_R\in C_c^\infty(D_{R+1})$ such that
$0\leq\chi_R\leq1$ on $M$
and 
$\chi_R\equiv1$ on an open neighborhood of $\overline{D_R}$.
We define
\[
F_R(r,y):=\chi_R(y)z(t-r,y),
\qquad (r,y)\in[0,S]\times M.
\]
The restriction $0<\varepsilon<t<T$ is important here.  Indeed, one may
choose $\eta>0$ so small that
$[\varepsilon-\eta,t+\eta]\subset(0,T)$, and hence $F_R$ extends smoothly to
$(-\eta,S+\eta)\times M$.  Moreover, its spatial support is contained in
the fixed compact set $\operatorname{supp}\chi_R\Subset D_{R+1}$.  We may
therefore apply \eqref{eq:intrinsic-time-dynkin} with $D=D_R$ and
$F=F_R$.

For $0\leq s\leq S$, put
\[
\theta_s^{(R)}:=s\wedge\tau_R,
\qquad
A_s^{(R)}:=\int_0^{\theta_s^{(R)}}V(X_r)\,dr,
\qquad
Y_s^{(R)}:=z\bigl(t-\theta_s^{(R)},X_{\theta_s^{(R)}}\bigr).
\]
By continuity of the sample paths of the (weighted) Brownian motion, 
$X_{\theta_s^{(R)}}\in\overline{D_R}$ holds.  Hence
we have
\[
F_R\bigl(\theta_s^{(R)},X_{\theta_s^{(R)}}\bigr)=Y_s^{(R)},
\qquad
F_R(0,x)=z(t,x).
\]
Furthermore, for $r<\tau_R$, we have $X_r\in D_R$. Therefore
$\chi_R(X_r)=1$,
$\nabla\chi_R(X_r)=0$ and 
$\Delta_\mu\chi_R(X_r)=0$ hold.
Using the product rule for $\Delta_\mu$, we consequently obtain
\begin{align*}
(\partial_r+\Delta_\mu)F_R(r,X_r)
&=
\bigl(-\partial_tz+\Delta_\mu z\bigr)(t-r,X_r)\\
&=V(X_r)z(t-r,X_r),
\qquad r<\tau_R,
\end{align*}
where the last equality follows from
$\partial_tz=\Delta_\mu z-Vz$.
Thus the stopped time-dependent Dynkin formula 
\eqref{eq:intrinsic-time-dynkin} yields a continuous
martingale $(N_s^{(R)})_{0\leq s\leq S}$ such that
\begin{align}
Y_s^{(R)}
&=z(t,x)
 +\int_0^s\1_{\{r<\tau_R\}}
 V(X_r)Y_r^{(R)}\,dr
 +N_s^{(R)},
\qquad 0\leq s\leq S,
\label{eq:Y-decomp}
\end{align}
where we used that $Y_r^{(R)}=z(t-r,X_r)$ on the set
$\{r<\tau_R\}$.

We now set
\[
U_s^{(R)}:=e^{-A_s^{(R)}}Y_s^{(R)},
\qquad 0\leq s\leq S.
\]
Since $A^{(R)}$ and $Y^{(R)}$ satisfy
\[
dA_s^{(R)}=
\1_{\{s<\tau_R\}}V(X_s)\,ds, \qquad 0\leq s \leq S,
\]
and \eqref{eq:Y-decomp}, respectively, it follows from It\^{o}'s formula that
\begin{align}
dU_s^{(R)}
&=e^{-A_s^{(R)}}\,dY_s^{(R)}
  -e^{-A_s^{(R)}}Y_s^{(R)}\,dA_s^{(R)}
\notag\\
&=
e^{-A_s^{(R)}}\,dY_s^{(R)}-{\bf 1}_{\{s<\tau_{R}\}}V(X_{s})Y_{s}^{(R)}e^{-A_{s}^{(R)}}ds
=
e^{-A_s^{(R)}}\,dN_s^{(R)}.
\label{eq:dF}
\end{align}
Therefore $U^{(R)}$ is a continuous local martingale.  On the other hand, since $0<e^{-A_s^{(R)}}\leq1$, 
we also have
\[
|U_s^{(R)}|
\leq |Y_s^{(R)}|
\leq\|z\|_{L^\infty((0,T)\times M)},
\qquad 0\leq s\leq S.
\]
Thus $U^{(R)}$ is a bounded continuous martingale.  Equivalently,
\begin{equation}\label{eq:ito-product}
M_s^{(R)}
:=e^{-A_s^{(R)}}Y_s^{(R)}-z(t,x),
\qquad 0\leq s\leq S,
\end{equation}
is a bounded continuous martingale.  Taking expectations at $s=S=t-\varepsilon$
gives
\begin{equation}
\label{FKF-2}
z(t,x)
=
\E_x\!\left[
\exp\!\Bigl(-\int_0^{(t-\varepsilon)\wedge\tau_R}V(X_r)\,dr\Bigr)
z\bigl(t-(t-\varepsilon)\wedge\tau_R,X_{(t-\varepsilon)\wedge\tau_R}\bigr)
\right].
\end{equation}

Finally, we recall stochastic completeness \eqref{SC-life}. Combining this with
\eqref{eq:nested-exhaustion}, we easily see 
${\mathbb P}_{x}\big(\lim_{R\to \infty}\tau_{R}=\infty \big)=1$.
Hence for the fixed finite time $t-\varepsilon$, $\lim_{R\to \infty}(t-\varepsilon)\wedge \tau_{R}=t-\varepsilon$
holds $\mathbb P_{x}$-almost surely. Since $z(\varepsilon, \cdot)$ is smooth and bounded, the dominated 
convergence theorem applied to \eqref{FKF-2} yields the desired formula \eqref{eq:stopped-FK}.
\end{proof}

For later use, assuming stochastic completeness of $M$,
we collect several standard facts about the \textit{Feynman--Kac semigroup}
associated with the potential $V$, defined by
$$T^{V}_{t}f(x):={\mathbb E}_{x}\Big[ \exp \big(-\int_{0}^{t}V(X_{s})\,ds \big) f(X_{t}) \Big], \quad f\in B_{b}(M).$$
\begin{lm}\label{FK-property}
Let $p^{V}(t,x,y)$ be the minimal Schr{\"o}dinger kernel function 
introduced in Section {\rm{2}}. Then we have
\begin{align}
T^{V}_{t}f(x)&=\int_{M} p^{V}(t,x,y) f(y) \mu(dy), \quad t>0, \, x\in M, \, f\in B_{b}(M).
\label{TV-kernel}
\end{align}
Moreover,
\begin{align}
p^{V}(t,x,y)&=p(t,x,y){\mathbb E}_{x,y;t}\big[ \exp \big(-\int_{0}^{t}V(X_{s})ds\big)\big], \quad t>0, \, x,y\in M,
\label{TV-kernel2}
\end{align}
where ${\mathbb E}_{x,y;t}$ denotes expectation under
the (weighted) Brownian bridge measure ${\mathbb P}_{x,y;t}$ on $C([0,t],M)$ with 
${\mathbb P}_{x,y;t}(X_{0}=x, X_{t}=y)=1$.
Consequently, we have
\begin{align} 0<p^{V}(t, x,y)\leq p(t,x,y), \quad t>0, \, x,y\in M.
\label{HK-comparison}
\end{align}
\end{lm}
\begin{proof} We denote the right-hand side of (\ref{TV-kernel}) by $P_{t}^{V}f(x)$ and set
$${\cal H}:=\{f\in B_{b}(M); \, T_{t}^{V}f(x)=P_{t}^{V}f(x), \, t>0, x\in M \}.$$
Repeating the same argument as in the proof of Lemma \ref{lem:ito-detailed}, we have
\eqref{TV-kernel} for $f\in C^{\infty}_{c}(M)$ (see also \cite[Theorem 3.2, Chapter V]{IW89}
for details). This means $C^{\infty}_{c}(M) \subset {\cal H}$.
If $\{f_{j}\}_{j=1}^{\infty}\subset {\cal H}$ satisfies $0\leq f_{j}(x) \nearrow f(x)$ for all $x\in M$ and some $f\in B_{b}(M)$, 
$T_{t}^{V}f_{j}(x) \nearrow T_{t}^{V}f(x)$ and $P_{t}^{V}f_{j}(x) \nearrow P_{t}^{V}f(x)$ hold for all $t>0$ and 
$x\in M$ by applying Lebesgue's monotone convergence theorem. Thus we have $f\in {\cal H}$.
By choosing a sequence $\{\chi_{j}\}_{j=1}^{\infty}\subset C^{\infty}_{c}(M)$ such that 
$0\leq \chi_{j}(x) \nearrow 1, x\in M$ as $j\to \infty$ and recalling $C^{\infty}_{c}(M) \subset {\cal H}$, we also have
$1\in {\cal H}$. Then the functional monotone class theorem (cf. \cite[Theorem 8.15]{Sch17}) implies that 
${\cal H}$ contains every bounded $\sigma(C^{\infty}_{c}(M))$-measurable function.
Since $\sigma(C^{\infty}_{c}(M))$ coincides with the Borel $\sigma$-algebra on $M$, 
this means ${\cal H}=B_{b}(M)$. Thus we have shown (\ref{TV-kernel}).

As in \cite[Section 5.4]{Hsu02}, we may rewrite $T_{t}^{V}f(x)$ as 
\begin{align*}
T_{t}^{V}f(x)&=\int_{M} {\mathbb E}_{x,y;t}\Big[ \exp \big(
-\int_{0}^{t}V(X_{s})ds \big) \Big] p(t,x,y) f(y) \mu(dy).
\end{align*}
Then we have (\ref{TV-kernel2}) as a consequence of (\ref{TV-kernel}).
Moreover, it follows from $V\geq 0$ that
$0<{\mathbb E}_{x,y;t}\big[ \exp \big(
-\int_{0}^{t}V(X_{s})ds \big) \big] \leq 1$. Combining this with (\ref{TV-kernel2}), we obtain (\ref{HK-comparison}).
\end{proof}
\subsection{Proof of Theorem~\ref{thm:uniq}}
In this subsection, we give a proof of Theorem~\ref{thm:uniq}. Throughout this subsection, we always
assume that $(M,g,\mu)$ is stochastically complete. As preparation, we show the following two lemmas.
\begin{lm}\label{kernel-L1}
For a fixed $x\in M$, we define a function $K^{V}_{x}:(0,\infty) \to L^{1}(\mu)$ by
$K^{V}_{x}(t):=p^{V}(t,x,\cdot)$ for $t>0$. Then $K^{V}_{x}$ is continuous.
\end{lm}
\begin{proof} We fix $t>0$ and $x\in M$ and take a sequence $(t_{j})_{j=1}^{\infty}$ in $(0,\infty)$
such that $t_{j}\to t$ as $j\to \infty$. We set 
$f_{j}(y):=\vert \, p^{V}(t_{j},x,y)-p^{V}(t,x,y) \vert$, $g_{j}(y)=p(t_{j},x,y)+p(t,x,y)$ and 
$g(y)=2p(t,x,y)$ for $y\in M$. The comparison bound \eqref{HK-comparison} gives
$$ f_{j}(y) \leq p^{V}(t_{j},x,y)+p^{V}(t,x,y) \leq g_{j}(y), \quad y\in M.$$
By stochastic completeness of $M$, we also have
$$\int_{M} g_{j}(y)\mu(dy)=2=\int_{M}g(y) \mu(dy), \quad j\in \mathbb N.
$$
Since $p^{V}$ and $p$ are smooth,
$\lim_{j\to \infty}f_{j}(y)=0$ and $\lim_{j\to \infty}g_{j}(y)=g(y)$. We therefore obtain
$\lim_{j\to \infty} \int_{M} f_{j}(y) \mu(dy)=0$ by
applying Pratt's lemma (cf. \cite[Problem 12.3]{Sch17}).
This means $\lim_{j\to \infty}\Vert K^{V}_{x}(t_{j})-K^{V}_{x}(t) \Vert_{L^{1}(\mu)}=0$, which completes the proof.
\end{proof}
\begin{lm}\label{lem:FK-sol}
Let $u_{0}:M\to [0,1]$ be a Borel measurable function and set
\begin{equation}\label{eq:FK-solution-at-zero}
u(0,x):=u_{0}(x),\qquad
u(t,x):=T^{V}_t u_{0}(x),\quad 0<t\leq T,\ x\in M.
\nonumber
\end{equation}
Then
$\pi^{u}:=\{\pi^{u}_{t}(dx):=u(t,x)\mu(dx) \}_{0\leq t\leq T}$
is a bounded weak solution to
\eqref{hdleq} in the sense of Definition~\ref{def:weak}.
\end{lm}
\begin{proof}
First, we prove $\pi^{u}\in C([0,T], \mc M_+(M))$.
We use the convention $T_0^V={\rm{Id}}$.
Using 
\eqref{TV-kernel} and
the symmetry of the minimal Schr{\"{o}}dinger kernel function $p^{V}(t,x,y)=p^{V}(t,y,x)$, $t>0, x,y\in M$, we easily 
have
\begin{align}
\int_{M} \varphi(x) (T^{V}_{t}u_{0})(x) \, \mu(dx)=\int_{M} u_{0}(x)(T^{V}_{t}\varphi)(x) \, \mu(dx), \quad \varphi \in 
C_{c}(M),~t\geq 0.
\label{TV-symmetric}
\end{align}
Since $T_{t}^{V}f=P^{V}_{t}f=e^{-tH_{V}}f$ for $f\in C^{\infty}_{c}(M)$,
$\{T_{t}^{V}\}_{t\geq 0}$ is regarded as a symmetric strongly continuous 
contraction semigroup on $L^{2}(\mu)$. 
Moreover, it can be extended to an $L^{1}(\mu)$-strongly continuous contraction semigroup
(see e.g., \cite[Proposition 2.2]{Shi97} for the proof).
Combining this with (\ref{TV-symmetric}), we see
\begin{align}
\big \vert \langle \pi^{u}_{t}, \varphi \rangle
- \langle \pi^{u}_{s}, \varphi \rangle \big \vert
&=\Big \vert \int_{M} \big( T_{t}^{V}\varphi(x)-T_{s}^{V}\varphi(x) \big) u_{0}(x) \mu(dx) \Big \vert
\nonumber \\
&\leq \Vert T_{t}^{V}\varphi-T_{s}^{V}\varphi \Vert_{L^{1}(\mu)}
\Vert u_{0} \Vert_{L^{\infty}(\mu)} \to 0  ~~ \mbox{ as }
s\to t, \quad \varphi \in C_{c}(M),
\nonumber 
\end{align}
which means $\pi^{u}\in C([0,T], \mc M_+(M))$.

Since $0\leq T^{V}_{t}u_{0}(x) \leq 1$ holds for all $x\in M$ and $0\leq t \leq T$, 
we have shown (ii) in Definition \ref{def:weak}. Hence it remains to prove
(\ref{weak heat}). 
The same argument as in the proof of Lemma \ref{lem:ito-detailed} gives
\begin{equation}
T^{V}_{t}\varphi(x)=\varphi(x)-\int_{0}^{t} T^{V}_{s}(H_{V}\varphi)(x) \, ds, \quad t\geq 0, \, x\in M, \, \varphi\in 
C^{\infty}_{c}(M).
\label{TV-FK-Ito}
\end{equation}
Multiplying \eqref{TV-FK-Ito} by $u_0(x)$,
integrating with respect to $\mu(dx)$, and using \eqref{TV-symmetric}, we
obtain
\begin{align*}
\langle\pi_t^u,\varphi\rangle-\langle\pi_0^u,\varphi\rangle
&=-\int_0^t\int_Mu_0(x)T_s^V(H_V\varphi)(x)\,\mu(dx)\,ds\\
&=-\int_0^t\int_M(T_s^Vu_0)(x)H_V\varphi(x)\,\mu(dx)\,ds
=\int_0^t\langle\pi_s^u,\Delta_\mu\varphi-V\varphi\rangle\,ds,
\end{align*}
and hence condition \textup{(iii)} is satisfied.
We complete the proof.
\end{proof}
\begin{proof}[{\bfseries Proof of Theorem~\ref{thm:uniq}}]
Let $\pi_t(dx)=u(t,x)\mu(dx)$ and 
$\widetilde\pi_t(dx)=v(t,x)\mu(dx)$
be two bounded weak solutions with the same initial Borel measurable function
$u_0:M\to[0,1]$.  Let $C_u$ and $C_v$ be positive constants for which condition
\textup{(ii)} of Definition~\ref{def:weak} holds for $\pi$ and
$\widetilde\pi$, respectively.
We set $f(t,x):=u(t,x)-v(t,x)$ and
$C_f:=C_u+C_v$.
It follows from Lemma~\ref{lem:interior-smooth}
that $f \in C^\infty((0,T)\times M)$
and $\|f(t, \cdot) \|_{L^\infty (\mu)}\le C_{f}$ for any $0\leq t \leq T$.  
Moreover, $f$ satisfies $f(0,x)=0$, $x\in M$ and 
\begin{equation}\label{eq:w-eq}
\partial_t f(t,x) = \Delta_\mu f(t,x) - V(x)f(t,x), \qquad t>0,~x\in M.
\end{equation}

We now fix $t\in(0,T)$, $x\in M$ and $\varepsilon\in(0,t)$.  By
Lemma~\ref{lem:ito-detailed}, we have
\begin{equation}\label{eq:mart-rep}
f(t,x)=\E_x\!\left[\exp\!\Bigl(-\int_0^{t-\varepsilon}V(X_r)\,dr\Bigr)
f(\varepsilon,X_{t-\varepsilon})\right].
\end{equation}
We may rewrite \eqref{eq:mart-rep} as 
\begin{equation}\label{eq:FK-semigroup}
f(t,x) = \big(T^{V}_{t-\varepsilon}\,f(\varepsilon,\cdot)\big)(x)=\int_{M} p^{V}(t-\varepsilon, x,y) f(\varepsilon,y) \mu(dy).
\end{equation}

For fixed $t>0$ and $x\in M$, we set $K^{V}_{x}(t):=p^{V}(t,x,\cdot)\in L^1(\mu)$.  
Since $C_c(M)$ is dense in $L^1(\mu)$, for any small $\delta>0$, we may choose a function 
$\phi_\delta\in C_c(M)$ such that
$
\|
K^{V}_{x}(t)
-\phi_\delta\|_{L^1(\mu)}<\delta/2
$.
Moreover, by Lemma~\ref{kernel-L1}, there exists $\eta\in(0,t/2)$ such that
$
\big \| 
K^{V}_{x}(t-\varepsilon)-K^{V}_{x}(t)
\big \|_{L^1(\mu)}<\delta/2
$
for all $\varepsilon\in(0,\eta)$,
and therefore
\[
\|K^{V}_{x}(t-\varepsilon)
-\phi_\delta\|_{L^1(\mu)}<\delta
\qquad\text{for all }\varepsilon\in(0,\eta).
\]
Now, for $0<\varepsilon<\eta$, \eqref{eq:FK-semigroup} and the bound
$\|f(\varepsilon,\cdot)\|_{L^{\infty}(\mu)}\le C_{f}$ imply
\begin{align}
|f(t,x)|
&=
\left|\int_M p^{V}(t-\varepsilon,x,y)\,f(\varepsilon,y)\,\mu(dy)\right|
\notag\\
&\le
\left|\int_M \phi_\delta(y)\,f(\varepsilon,y)\,\mu(dy)\right|
+
\| 
K^{V}_{x}(t-\varepsilon)
-\phi_\delta\|_{L^1(\mu)}
\|f(\varepsilon,\cdot)\|_{L^\infty}
\notag\\
&\le
\left|\int_M \phi_\delta(y)\,f(\varepsilon,y)\,\mu(dy)\right|
+C_{f}\delta.
\label{eq:split}
\end{align}

Next, we show that the first term on the right-hand side of \eqref{eq:split} converges
to $0$ as $\varepsilon\searrow0$.
Recall that the signed measure $\pi^{f}_\varepsilon(dx):=f(\varepsilon,x)\,\mu(dx)$ satisfies
$\pi^{f}_\varepsilon(dx)=\pi_{\varepsilon}(dx)-{\tilde{\pi}}_{\varepsilon}(dx)$, where
$\pi, {\tilde{\pi}} \in C([0,T];\Mplus)$ with the same initial datum $u_0(x)\,\mu(dx)$.
By the definition of vague continuity, for every $\varphi\in C_c(M)$, we have
\begin{align*}
\int_M \varphi(x) \pi^{f}_\varepsilon (dx)
&= \int_M\varphi(x)\pi_\varepsilon (dx)-
\int_M\varphi(x){\tilde{\pi}}_\varepsilon (dx)
\nonumber \\
& \longrightarrow
\int_M\varphi(x)\,
\pi_{0} (dx)
-
\int_M\varphi(x)\,
{\tilde{\pi}}_{0} (dx)=0
\quad 
\mbox{ as } \varepsilon \searrow 0.
\end{align*}
In particular, taking $\varphi=\phi_\delta\in C_c(M)$, we obtain
\[
\lim_{\varepsilon \searrow 0} \int_M \phi_\delta(y)\,f(\varepsilon,y)\,\mu(dy)=0.
\]
Hence (\ref{eq:split}) yields
\[
|f(t,x)| \le 0 + C_{f}\delta, \qquad t\in (0,T), \, x\in M.
\]
Since $\delta>0$ is arbitrary, 
$f\equiv 0$ on $(0,T)\times M$. Moreover, combining this with $f(0,\cdot)=0$, 
we have $\pi_t=\widetilde\pi_t$ for all $0\leq t<T$.

It remains to show $\pi_T=\widetilde\pi_T$.  Let
$(t_n)_{n\geq1}\subset(0,T)$ satisfy $t_n\nearrow T$.  For every
$\varphi\in C_c(M)$, the vague continuity of both solutions $\pi$ and $\widetilde \pi$ and
$\pi_{t}={\widetilde \pi}_{t}$, $0\leq t<T$, 
imply
\begin{align*}
\langle\pi_T,\varphi\rangle
&=\lim_{n\to\infty}\langle\pi_{t_n},\varphi\rangle
 =\lim_{n\to\infty}\langle\widetilde\pi_{t_n},\varphi\rangle
 =\langle\widetilde\pi_T,\varphi\rangle.
\end{align*}
This means $\pi_T=\widetilde\pi_T$, and thus we have shown
$\pi_t=\widetilde\pi_t$ for all $0\leq t\leq T$.

Finally, by Lemma~\ref{lem:FK-sol}, the unique solution is given by
$\pi_{t}(dx)=(T^{V}_{t}u_{0})(x) \mu(dx)$, $0\leq t \leq T$.
This yields the representation formula
\eqref{explicit-representation} and we complete the proof.
\end{proof}
\vspace{2mm}
\noindent
{\textbf{Acknowledgments.}}~The first author was partially supported by JSPS Grant-in-Aid for Scientific Research (C)
No. 22K03280, No. 26K06833 and (S) No. 22H04942. The second author was partially supported by JSPS Grant-in-Aid for 
Scientific Research (C) No. 23K03155, 26K06833, (B) No. 23K20801 and 
the Fukuzawa Fund (Keio Gijuku Fukuzawa Memorial
Fund for Advancement of Education and Research) of Keio University and the visiting fellowships of Exeter College,
University of Oxford. The third author was partially supported by 
JSPS Grant-in-Aid for Scientific Research (C) No. 26K06832.
\vspace{2mm} \\
\noindent
{\bf Data availability.}
No data are associated with this article.
\vspace{5mm} \\
\noindent
{\large{\bf Declarations}}
\vspace{2mm} \\
\noindent
{\bf Conflict of interest.}
The authors declare that they have no conflicts of interest.


\begin{thebibliography}{99}
%
%
%
\bibitem[CG21]{CG21}
J.P. Chen and P. Gon\c{c}alves:
{\it{Asymptotic behavior of density in the boundary-driven exclusion process on the Sierpinski gasket}},
Math. Phys. Anal. Geom. {\bf{24}} (2021), pp. 1--65.
%
\bibitem[EK86]{EK86}
S.N. Ethier and T.G. Kurtz:
{\it{Markov Processes: Characterization and Convergence}},
Wiley Series in Probability and Mathematical Statistics,
John Wiley \& Sons, New York, 1986, x+534 pp.
%
\bibitem[Fag10]{Fag10}
A. Faggionato:
{\it{Hydrodynamic limit of symmetric exclusion processes in inhomogeneous media}},
{\tt{arXiv:1003.5521}} (2010), 8 pp.
%
\bibitem[Fun18]{Fun18}
T. Funaki:
{\it{Hydrodynamic limit for exclusion processes}},
Commun. Math. Stat. {\bf{6}} (2018), pp. 417--480.
%
\bibitem[GR20]{GR20}
B. van Ginkel and F. Redig:
{\it{Hydrodynamic limit of the symmetric exclusion process on a compact Riemannian manifold}},
J. Stat. Phys. {\bf{178}} (2020), pp. 75--116.
%
\bibitem[GR22]{GR22}
B. van Ginkel and F. Redig:
{\it{Equilibrium fluctuations for the symmetric exclusion process on a compact Riemannian manifold}},
Markov Proc. Relat. Fields {\bf{28}} (2022), pp. 29--51.
%
\bibitem[Gri06]{Gri06}
A. Grigor'yan:
{\it{Heat kernels on weighted manifolds and applications}},
in: The Ubiquitous Heat Kernel, Contemp. Math. {\bf{398}},
American Mathematical Society, Providence, RI, 2006, pp. 93--191.
%
\bibitem[Gri09]{Gri09}
A. Grigor'yan:
{\it{Heat Kernel and Analysis on Manifolds}},
AMS/IP Studies in Advanced Mathematics {\bf{47}},
American Mathematical Society, Providence, RI, 2009, xviii+482 pp.
%
\bibitem[Gua23]{Gua23}
Z. Guan:
{\it{Hydrodynamic limits of interacting particle systems on crystal lattices in periodic realizations}},
Electron. J. Probab. {\bf{28}} (2023), Paper No. 97, 30 pp.
%
\bibitem[Hsu02]{Hsu02}
E.P. Hsu:
{\it{Stochastic Analysis on Manifolds}},
Graduate Studies in Mathematics {\bf{38}},
American Mathematical Society, Providence, RI, 2002, xiv+281 pp.
%
\bibitem[IW89]{IW89}
N. Ikeda and S. Watanabe:
{\it{Stochastic Differential Equations and Diffusion Processes}}, 2nd ed.,
North-Holland Mathematical Library {\bf{24}}, North-Holland Publishing Co.,
Amsterdam; Kodansha, Ltd., Tokyo, 1989, xvi+555 pp.
%
\bibitem[IK24]{IK24}
S. Ishiwata and H. Kawabi:
{\it{A graph discretized approximation of semigroups for diffusion with drift and killing
on a complete Riemannian manifold}}, Math. Ann. {\bf{390}} (2024), pp. 2459--2495.
%
\bibitem[Jak86]{Jak86}
A. Jakubowski:
{\it{On the Skorokhod topology}},
Ann. Inst. Henri Poincar\'{e} Probab. Stat. {\bf{22}} (1986), pp. 263--285.
%
\bibitem[Jar09]{Jar09}
M. Jara: 
{\it{Hydrodynamic limit for a zero-range process in the Sierpinski gasket}},
Comm. Math. Phys. {\bf{288}} (2009), pp. 773--797.
%
\bibitem[Jar11]{Jar11}
M. Jara:
{\it{Hydrodynamic limit of the exclusion process in inhomogeneous media}},
in: Dynamics, Games and Science II, Springer Proc. Math. {\bf{2}},
Springer, Heidelberg, 2011, pp. 449--465.
%
\bibitem[JRV24]{JRV24}
J. Junn\'e, F. Redig and R. Versendaal:
{\it{Hydrodynamic limit of the symmetric exclusion process on complete Riemannian manifolds and 
principal bundles}},
{\tt{arXiv:2410.20167v2}} (2024), 28 pp.
%
\bibitem[Kal17]{Kal17}
O. Kallenberg:
{\it{Random Measures, Theory and Applications}},
Probability Theory and Stochastic Modelling {\bf{77}},
Springer, Cham, 2017, xiii+694 pp.
%
\bibitem[Kan85]{Kan85}
M. Kanai:
{\it{Rough isometries, and combinatorial approximations of geometries of non-compact
Riemannian manifolds}},
J. Math. Soc. Japan {\bf{37}} (1985), pp. 391--413.
%
\bibitem[KL99]{KL99}
C. Kipnis and C. Landim:
{\it{Scaling Limits of Interacting Particle Systems}},
Grundlehren der Mathematischen Wissenschaften {\bf{320}},
Springer-Verlag, Berlin, 1999, xvi+442 pp.
%
\bibitem[KS00]{KS00}
M. Kotani and T. Sunada: 
{\it{Albanese maps and off diagonal long time asymptotics for the heat kernel}}, 
Comm. Math. Phys. {\bf{209}} (2000), pp. 633--670.
%
\bibitem[Lig85]{Lig85}
T.M. Liggett:
{\it{Interacting Particle Systems}},
Grundlehren der Mathematischen Wissenschaften {\bf{276}},
Springer-Verlag, Berlin, 1985, xv+488 pp.
%
\bibitem[vMT26]{vMT26}
P. van Meurs and K. Tsunoda:
{\it{Hydrodynamic limit for Glauber--Kawasaki dynamics on the Sierpi\'{n}ski gasket}},
{\tt{arXiv:2602.19059}} (2026), 39 pp.
%
\bibitem[Sal01]{Sal01}
L. Saloff-Coste:
{\it{Probability on groups: random walks and invariant diffusions}}, 
Notices Amer. Math. Soc. {\bf{48}} (2001), pp. 968--977.
%
\bibitem[Sch17]{Sch17}
R.L. Schilling:
{\it{Measures, Integrals and Martingales}}, 2nd ed.,
Cambridge University Press, Cambridge, 2017, xvii+476 pp.
%
\bibitem[Sep08]{Sep08}
T. Sepp\"{a}l\"{a}inen:
{\it{Translation invariant exclusion processes}}, lecture notes,
available at {\url{https://people.math.wisc.edu/~tseppalainen/excl-book/ajo.pdf}},
2008, 223 pp.
%
\bibitem[Shi97]{Shi97}
I. Shigekawa:
{\it{$L^{p}$ contraction semigroups for vector valued functions}},
J. Funct. Anal. {\bf{147}} (1997), pp. 69--108.
%
\bibitem[Spo91]{Spo91}
H. Spohn:
{\it{Large Scale Dynamics of Interacting Particles}},
Springer, Berlin--Heidelberg, 1991, xi+342 pp.
%
\bibitem[Sun08]{Sun08}
T. Sunada: 
{\it{Discrete geometric analysis}}, 
in: Analysis on Graphs and its Applications, 
Proc. Sympos. Pure Math. {\bf{77}},
Amer. Math. Soc., Providence, RI, 2008, pp. 51--83.
%
\bibitem[Sun13]{Sun13}
T. Sunada:
{\it{Topological Crystallography: With a View Towards Discrete Geometric Analysis}},
Surveys and Tutorials in the Applied Mathematical Sciences {\bf{6}},
Springer, Tokyo, 2013, xii+229 pp.
%
\bibitem[Tan12]{Tan12}
R. Tanaka:
{\it{Hydrodynamic limit for weakly asymmetric simple exclusion processes in crystal lattices}},
Comm. Math. Phys. {\bf{315}} (2012), pp. 603--641.
%
\bibitem[Tre67]{Tre67}
F.~Tr\`{e}ves:
{\it{Topological Vector Spaces, Distributions and Kernels}},
Academic Press, New York, 1967, xvi+624 pp.
%
\bibitem[Woe05]{Woe05}
W. Woess:
{\it{Lamplighters, Diestel--Leader graphs, random walks, and harmonic functions}},
Combinatorics, Probability \& Computing {\bf{14}} (2005), pp. 415--433.
\end{thebibliography}
\end{document}